\pdfoutput=1
\documentclass[oneside, hidelinks, 12pt, a4paper,linktoc=page,colorlinks,reqno,authoryear]{amsart}
\usepackage[toc,page]{appendix}
\usepackage{amsthm,amssymb,mathtools,tensor}
\usepackage{titling}
\usepackage[T1]{fontenc}
\usepackage[bitstream-charter]{mathdesign}
\usepackage{XCharter}
\let\mathcal\undefined
\DeclareMathAlphabet{\mathcal}{U}{dutchcal}{m}{n}
\usepackage[utf8]{inputenc}
\usepackage{enumerate,comment,braket,xspace,tikz-cd,csquotes}
\usepackage[driver=dvips,centering,vscale=0.7,hscale=0.8]{geometry}

\usepackage{microtype}
\usepackage[greek.polutoniko,english]{babel}
\usepackage[backend=biber,sorting=nyt,citestyle=ieee-alphabetic, bibstyle=alphabetic,texencoding=ascii,bibencoding=utf8,maxnames=5,minnames=3,maxbibnames=99]{biblatex}
\usepackage{setspace}
\usepackage{enumitem}
\usepackage[cal=dutchcal, scr=boondoxo]{mathalfa}
\usepackage{listings}
\usepackage{colordvi}
\usepackage{tikz,pgf}
\usetikzlibrary{arrows,decorations.pathmorphing,backgrounds,positioning,fit,matrix,calc}
\tikzset{curve/.style={settings={#1},to path={(\tikztostart)
			.. controls ($(\tikztostart)!\pv{pos}!(\tikztotarget)!\pv{height}!270:(\tikztotarget)$)
			and ($(\tikztostart)!1-\pv{pos}!(\tikztotarget)!\pv{height}!270:(\tikztotarget)$)
			.. (\tikztotarget)\tikztonodes}},
	settings/.code={\tikzset{quiver/.cd,#1}
		\def\pv##1{\pgfkeysvalueof{/tikz/quiver/##1}}},
	quiver/.cd,pos/.initial=0.35,height/.initial=0}
\tikzset{tail reversed/.code={\pgfsetarrowsstart{tikzcd to}}}
\tikzset{2tail/.code={\pgfsetarrowsstart{Implies[reversed]}}}
\tikzset{2tail reversed/.code={\pgfsetarrowsstart{Implies}}}
\usepackage{epsfig}
\usepackage{epstopdf}
\usepackage{tikz-cd}
\usepackage{fancyhdr}
\usepackage{graphicx}
\usepackage{fullpage}
\usepackage{hyperref}
\hypersetup{
	colorlinks=true,
	linkcolor=red,
	anchorcolor=black,
	filecolor=blue,
	citecolor = blue,   
	menucolor=.,   
	urlcolor=cyan,
	breaklinks=true,
}
\usepackage[capitalize]{cleveref}
\usepackage{nameref}
\usepackage{enumitem}
\crefalias{enumi}{lemman}
\usepackage{comment}
\usepackage{color}
\usepackage[totoc]{idxlayout}
\usepackage[lastexercise]{exercise}
\usepackage{xcolor}
\usepackage{breakcites}
\usetikzlibrary{positioning}

\makeatletter
\newcommand*{\doublerightarrow}[2]{\mathrel{
		\settowidth{\@tempdima}{$\scriptstyle#1$}
		\settowidth{\@tempdimb}{$\scriptstyle#2$}
		\ifdim\@tempdimb>\@tempdima \@tempdima=\@tempdimb\fi
		\mathop{\vcenter{
				\offinterlineskip\ialign{\hbox to\dimexpr\@tempdima+1em{##}\cr
					\rightarrowfill\cr\noalign{\kern.5ex}
					\rightarrowfill\cr}}}\limits^{\!#1}_{\!#2}}}
\newcommand*{\triplerightarrow}[1]{\mathrel{
		\settowidth{\@tempdima}{$\scriptstyle#1$}
		\mathop{\vcenter{
				\offinterlineskip\ialign{\hbox to\dimexpr\@tempdima+1em{##}\cr
					\rightarrowfill\cr\noalign{\kern.5ex}
					\rightarrowfill\cr\noalign{\kern.5ex}
					\rightarrowfill\cr}}}\limits^{\!#1}}}
\makeatother

\makeatletter
\def\@tocline#1#2#3#4#5#6#7{\relax
	\ifnum #1>\c@tocdepth % then omit
	\else
	\par \addpenalty\@secpenalty\addvspace{#2}%
	\begingroup \hyphenpenalty\@M
	\@ifempty{#4}{%
		\@tempdima\csname r@tocindent\number#1\endcsname\relax
	}{%
		\@tempdima#4\relax
	}%
	\parindent\z@ \leftskip#3\relax \advance\leftskip\@tempdima\relax
	\rightskip\@pnumwidth plus4em \parfillskip-\@pnumwidth
	#5\leavevmode\hskip-\@tempdima
	\ifcase #1
	\or\or \hskip 1em \or \hskip 2em \else \hskip 3em \fi%
	#6\nobreak\relax
	\dotfill\hbox to\@pnumwidth{\@tocpagenum{#7}}\par
	\nobreak
	\endgroup
	\fi}
\makeatother

\newcommand{\infinity}{\mbox{\footnotesize $\infty$}}
\newcommand{\LL}{\mathbb{L}}

\newcommand{\triv}{\operatorname{triv}}
\newcommand{\oblv}{\operatorname{oblv}}
\newcommand{\fil}{\operatorname{fil}}
\newcommand{\BGagm}{\mathsf{B}\Ga\rtimes \Gm}

\newcommand{\Einf}{\mathbb{E}_{\infty}}

\newcommand{\Rep}{{\operatorname{Rep}}}

\newcommand{\op}{\operatorname{op}}

\newcommand{\Perf}{{\operatorname{Perf}}}

\newcommand{{\Cn}}{\operatorname{Cn}}

\newcommand{\BGa}{\mathsf{B}\Ga}
\newcommand{\BGm}{\mathsf{B}\Gm}

\newcommand{\colim}{\operatorname{colim}}

\newcommand{\Mapin}{{\underline{\smash{\Map}}}}

\newcommand{\scrC}{\mathscr{C}}

\newcommand{\hsp}{\hspace{0.1cm}}

\newcommand{\fib}{{\operatorname{fib}}}
\newcommand{\cofib}{{\operatorname{cofib}}}
\newcommand{\Spec}{{\operatorname{Spec}}}

\newcommand{\Qcoh}{{\operatorname{QCoh}}}

\newcommand{\lpd}{(\!(}
\newcommand{\rpd}{)\!)}

\newcommand{\Mod}{{\operatorname{Mod}}}

\newcommand{\lp}{\left(}
\newcommand{\rp}{\right)}

\newcommand{\Ind}{{\operatorname{Ind}}}

\newcommand*{\longhookrightarrow}{\ensuremath{\lhook\joinrel\relbar\joinrel\rightarrow}}

\makeatletter
\usepackage{etoolbox}
   \makeatletter

\appto\maketitle{%
	\let\@makefnmark\relax  \let\@thefnmark\relax
	\ifx\@empty\addresses\else\@footnotetext{%
		\vskip-\bigskipamount\@setaddresses}
}
\def\enddoc@text{}
\makeatother
\newcounter{savedchapter}% for remembering the last chapter number
\preto\appendix{\setcounter{savedchapter}{\arabic{chapter}}}% remembering!
\newcommand\resumechapters{% the \appendix command with some tweaks
	\setcounter{chapter}{\arabic{savedchapter}}% restore chapter number
	\setcounter{section}{0}% reset section counter
	\gdef\@chapapp{\chaptername}% reset chapter name
	\gdef\thechapter{\@arabic\c@chapter}% make chapter numbers arabic
}
\makeatother

\newcommand{\dgMod}{{\operatorname{dgMod}}}

\newcommand{\Gr}{{\operatorname{Gr}}}

\newcommand{\Ga}{\mathbb{G}_{\operatorname{a},\Bbbk}}
\newcommand{\Gm}{\mathbb{G}_{\operatorname{m},\Bbbk}}

\newcommand{\ZZ}{\mathbb{Z}}
\newcommand{\AAff}{\mathbb{A}}
\newcommand{\NN}{\mathbb{N}}

\newcommand{\QQ}{\mathbb{Q}}

\newcommand{\Map}{{\operatorname{Map}}}

\newcommand{\Sym}{{\operatorname{Sym}}}

\newcommand{\gfrak}{\mathfrak{g}}

\newcommand{\Fun}{{\operatorname{Fun}}}

\newcommand{\dR}{{\operatorname{dR}}}

\newcommand{\LX}{{\mathscr{L}X}}

\newcommand{\Modgr}{\Mod^{{\operatorname{gr}}}}
\newcommand{\gr}{{\operatorname{gr}}}
\newcommand{\Modmixgr}{\varepsilon\operatorname{-}\Mod^{{\operatorname{gr}}}}
\newcommand{\Modmixgrcn}{\varepsilon\operatorname{-}\Mod^{{\operatorname{gr},\geqslant0}}}

\newcommand{\mixgr}{\varepsilon\operatorname{-gr}}
\newcommand{\taufil}{\tau^{\fil}}

\newcommand{\Modfil}{\Mod^{{\operatorname{fil}}}}
\newcommand{\Modfilcomp}{\widehat{\Mod}^{\fil}}
\newcommand{\otimesfil}{\otimes^{\fil}}
\newcommand{\otimesmixgr}{\otimes^{\varepsilon\text{-}\gr}}

\newcommand{\Mapmixgr}{{\Mapin}^{\mixgr}}
\numberwithin{equation}{subsection}
\theoremstyle{plain}
\newtheorem{theorem}[equation]{Theorem}
\newtheorem{lemman}[equation]{Lemma}
\newtheorem{propositionn}[equation]{Proposition}
\newtheorem{corollaryn}[equation]{Corollary}

\newtheorem*{theoremn}{Theorem}

\theoremstyle{definition}
\newtheorem{defn}[equation]{Definition}
\newtheorem{parag}[equation]{}

\newtheorem{remark}[equation]{Remark}
\newtheorem{construction}[equation]{Construction}
\newtheorem{porism}[equation]{Porism}
\newtheorem{warning}[equation]{Warning}
\newtheorem{notation}[equation]{Notation}

\newcommand{\Sp}{{{\mathscr{S}}{\operatorname{p}}}}

\newcommand\restr[2]{{
		\left.\kern-\nulldelimiterspace
		#1
		\vphantom{\|}
		\right|_{#2} 
}}
\usepackage{lipsum}
\usepackage{fancyhdr}
\usepackage{authblk}
\providecommand{\abstract}{}
\usepackage{abstract}
\title{A $t$-structure on the $\infty$-category of mixed graded modules}
\author{Emanuele Pavia}
\address{Università degli Studi di Milano}
\email{emanuele.pavia@unimi.it}
\date{\today}
\DeclareGraphicsExtensions{.png,.jpg,.pdf,.mps}
\usepackage{fancyhdr}
\usepackage{lastpage}
\begin{document}
	\maketitle
	\begin{abstract}
	In this work, we shall study in a purely model-independent fashion the \infinity-category of mixed graded modules over a ring of characteristic $0$, as defined in \cite{PTVV} and \cite{CPTVV}, and collect some basic results about its main formal properties. Finally, we shall endow such \infinity-category with a both left and right complete accessible $t$-structure, showing how this identifies the \infinity-category of mixed graded modules with the left completion of the Beilinson $t$-structure on the \infinity-category of filtered modules. Most of the content of this paper is already available in literature, and it serves mainly as a reference for the work of \cite{pavia2}.
	\end{abstract}
	\tableofcontents
\section*{Introduction}
\addtocontents{toc}{\protect\setcounter{tocdepth}{0}}
%\addcontentsline{toc}{section}{Introduction}
\subsection*{Motivations}
It is well known that, in characteristic $0$, an action of the circle $S^1\coloneqq\mathsf{B}\ZZ$ (seen as the classifying derived stack for the constant stack $\ZZ$) over a derived stack $X$  is linked to differential forms and de Rham theory over $X$. Namely, in \cite{BZN} the authors established that $p$-forms over a derived stack $X$ can be interpreted via functions on the derived loop stack $\LX$, which is canonically endowed with an action of the circle $S^1$ by rotating loops. The condition of being closed is then encoded in the property that the function is $S^1$-equivariant. However, the homotopy theory of $S^1$-complexes is equivalent to the homotopy theory of \textit{mixed graded complexes}, i.e., chain complexes $M_{\bullet}$ endowed with a decomposition$$M_{\bullet}\coloneqq\bigoplus_{p\in\ZZ}M_p,$$where each $M_p$ is a sub-complex of $M_{\bullet}$, and endowed with a \textit{mixed differential} $\varepsilon_p\colon M_p\to M_{p-1}[-1]$ satisfying the usual square-to-zero property. The theory of mixed graded complexes in characteristic $0$, which was developed in \cite{PTVV} and \cite{CPTVV}, has been exploited extensively in the past years, and has also been linked to the usual derived filtered category of Beilinson in \cite{UHKR}, \cite{toen2020algebraic, derivedfoliations2} and \cite{calaque2021lie}.  Indeed, the de Rham algebra of a differential graded commutative ring $A_{\bullet}$, with its grading given by$$\dR^p{\lp A_{\bullet}\rp}=\Sym^p_{\Bbbk}{\lp\LL_{A_{\bullet}/\Bbbk}[1]\rp},$$is a \textit{mixed graded commutative algebra}, where the mixed differential is provided exactly by the de Rham differential. One can then define (shifted) $p$-forms and closed $p$-forms on a differential graded commutative algebra $A_{\bullet}$ in terms of elements of $\dR^p{\lp A_{\bullet}\rp}$ and homotopy fixed points for the de Rham differentials in $\dR^p{\lp A_{\bullet}\rp}$, respectively. Moreover, all these constructions satisfy descent, and make perfectly sense also for more general derived stacks.\\

This theoretical framework yields new perspectives over derived symplectic geometry and deformation quantization. Using mixed graded complexes and mixed graded cdga's one can define:
\begin{enumerate}
	\item shifted symplectic forms over a derived stack as shifted closed $2$-forms which are non-degenerate in some suitable sense;
	\item Lagrangian structures on morphisms of derived stacks;
	\item Poisson structures and their deformation quantization;
	\item derived algebraic foliations. 
\end{enumerate}
In the last years, mixed graded complexes have been employed also in the homotopy theory of Lie algebras and Lie algebroids. Given a differential graded Lie algebra $\gfrak_{\bullet}$, it is known that the Chevalley-Eilenberg algebra has a richer structure of \textit{mixed graded commutative algebra} (see also \cite{CG}, \cite{nuiten19}). In particular, it is expected that mixed graded complexes can provide a natural setting where to work with formal geometry and deformation theory - this last particular application is what we are most interested in. 
\subsection*{Outline of the paper}
The content of this paper stems from our work in the context of derived Lie algebras, and in particular from the study of the mixed graded structure of the Chevalley-Eilenberg algebras and coalgebras of Lie algebras, which is studied in greater detail in \cite{pavia2}. For this purpose, the main aim of this paper is to collect some of the most important and useful features of the stable \infinity-category $\Modmixgr_{\Bbbk}$ of mixed graded modules. In \cref{chapter:mixedgradedmodules} we investigate the computation of limits and colimits of mixed graded modules, their closed monoidal structure, and the relationship of mixed graded modules with purely graded modules and non-graded modules (\cref{sec:basicmixedgraded}). While the statements and the proofs are presented in a model-independent fashion, differently from what one can read in the existing literature, most of the results here gathered are far from being original. In particular, most of them can be found or easily derived from the content of \cite{PTVV}, \cite{CPTVV} and \cite{derivedfoliations2}. The main new contribution of this part is, arguably, the characterization of fully dualizable objects of $\Modmixgr_{\Bbbk}$ (\cref{prop:dualizableobjects}).\\
% and the proof that the symmetric co/algebra \infinity-functor, seen as a \textit{graded} co/algebra, is fully faithful (\cref{prop:gradedsymfullyfaithful} and \cref{remark:gradedsymalgebrafullyfaithful}). \\

In \cref{sec:filtered/mixedgraded}, we study the relationship between the \infinity-categories $\Modmixgr_{\Bbbk}$ and $\Modfil_{\Bbbk}$, where the latter denotes the \infinity-category of filtered modules. The main result of this paper is the characterization of $\Modmixgr_{\Bbbk}$ as the full sub-\infinity-category of $\Modfil_{\Bbbk}$ spanned by filtered modules with complete filtration. While this result has already been proved (see for example \cite{derivedfoliations2} and \cite{calaque2021lie}), our work offers a deeper insight on such embedding by taking into account some $t$-structures on both \infinity-categories, namely:
\begin{theoremn}[\cref{thm:mixedgradedtstructure,thm:leftcompletion}]
	%	\todo{Robalo, rendere più chiaro}
	There exists a left and right complete $t$-structure on the stable \infinity-category of mixed graded modules whose heart is equivalent to the usual abelian $1$-category of chain complexes. Moreover, the embedding of the \infinity-category of mixed graded modules into the \infinity-category of filtered modules admits a left adjoint $(-)^{\mixgr}\colon\Modfil_{\Bbbk}\to\Modmixgr_{\Bbbk}$ which identifies the \infinity-category of mixed graded $\Bbbk$-modules with the left completion $\Modfilcomp_{\Bbbk}$ of the Beilinson $t$-structure of \cite{beilinsonderived} on the \infinity-category of filtered $\Bbbk$-modules (which, in virtue of \cref{prop:leftcompletionBeilinson}, is the full sub-\infinity-category of modules with complete filtration).
\end{theoremn}
\subsection*{Acknowledgments}I am very thankful to my PhD advisor, M. Porta, who introduced me to the theory of mixed graded complexes. I would also like to thank S. Ariotta, F. Battistoni, I. Di Liberti, A. Gagna, and G. Nocera for fruitful discussions and suggestions which helped shaping the content of this pamphlet. Finally I would like D. Calaque and M. Robalo for precious comments and their careful refereeing.  
\subsection*{Notations, conventions and main references}\
%\addcontentsline{toc}{subsection}{Notations, conventions and main references}\
\begin{itemize}
	\item Throughout all this paper, we employ freely the language of derived algebraic geometry, \infinity-categories, and homotopical algebra provided by \cite{htt} and \cite{ha}, from which we borrow the formalism and most notations. Our language is \textit{innerly} derived: every definition and construction has to be interpreted, without further indication suggesting the contrary, in the context of higher algebra. In particular, by \textit{module} over a discrete commutative ring $\Bbbk$ we mean an object of the stable derived \infinity-category of $\Bbbk$-modules, by \textit{limits and colimits} we mean homotopy limits and colimits, by \textit{tensor product} we mean derived tensor product, and so forth.
	\item Our main references for the homotopy theory of mixed graded complexes are provided by \cite{PTVV} and \cite{CPTVV}.
	\item Our main references for the homotopy theory of Lie algebras in characteristic $0$, and its relationship with the derived deformation theory, are provided by \cite{dagx} and \cite{studyindag2}.
	\item When dealing with explicit models provided by chain complexes of $\Bbbk$-modules, we use a homological notation. 
	\item Our standing assumption is that we work in characteristic $0$, over a fixed commutative ring  $\Bbbk$.
	\item Throughout this paper, we shall often work with closed symmetric monoidal \infinity-categories $\scrC^{\otimes}$ enriched over $\Bbbk$-modules. In particular, such \infinity-categories are endowed with both internal mapping objects, obtained as a right adjoint to $\otimes$, and mapping $\Bbbk$-modules providing the enrichment over $\Mod_{\Bbbk}$. In order to avoid confusion, we shall denote the former with $\underline{\smash{\Map}}$ and the latter with $\Map$.
\end{itemize}
\addtocontents{toc}{\protect\setcounter{tocdepth}{2}}
\section{Mixed graded modules}
%\addcontentsline{toc}{chapter}{Mixed graded modules}
%\chaptermark{Mixed graded modules}
\label{chapter:mixedgradedmodules}
The fundamental objects of study in this article are \textit{mixed graded $\Bbbk$-modules}, which generalize the concept of \textit{mixed complexes} and provide - at least in characteristic $0$ - a very useful analogue to complexes endowed with a complete and exhaustive filtration. Mixed graded modules have been studied extensively in the last years in the field of derived differential geometry and theory of Lie algebroids; yet, they are not as well known as filtered $\Bbbk$-modules, of which they provide a more well-behaved analogue in characteristic $0$. In this first section, we first gather some important definitions and properties of the $\infinity$-category of mixed graded modules, and fix our notations. For this scope, our main sources are \cite{PTVV} and \cite{CPTVV}.
\subsection{Basic definitions and notations}
%\addcontentsline{toc}{section}{Basic definitions and notations}
\label{sec:basicmixedgraded}
In order to capture the idea behind the notion of a mixed graded $\Bbbk$-module, we first recall the concept of \textit{mixed $\Bbbk$-modules} (or \textit{mixed complexes}, as they are classically called), which - in the words of \cite{cyclichomology}, are objects that are \textit{both chain and cochain complexes in a compatible way}. Mixed complexes were first introduced in \cite{burghelea}, as \textit{algebraic $S^1$-chain complexes} (or \textit{chain complexes with an algebraic circle action}), in order to study Hochschild and cyclic homology of unital associative algebras in characteristic $0$, which naturally come equipped with a mixed structure at the level of chains. 
\begin{defn}[Mixed complexes, \cite{cyclichomology}]
	\label{def:mixedcomplexesclassical}
	A \textit{mixed complex} over a base ring $\Bbbk$ of characteristic $0$ is a chain complex $\lp C_{\bullet},\hsp\partial_{\bullet}\rp$ together with morphisms $\beta_{n}\colon C_{n}\to C_{n+1}$ such that $\beta_{n+1}\circ \beta_n=\partial_{n+1}\circ\beta_n+\beta_{n-1}\circ\partial_n=0$.
\end{defn}Equivalently, mixed complexes are modules over the free differential graded commutative algebra $\Bbbk[\eta]\coloneqq\Bbbk[t]/(t^2)$, where $\eta\coloneqq \bar{t}$ is a generator in homological degree $1$ and $\partial(\eta)=0$. Alternatively, they are comodules over the differential graded cocommutative coalgebra $\Bbbk[\varepsilon]=\lp\Bbbk[\eta]\rp^{\vee}$ which is the (differential graded) $\Bbbk$-linear dual of $\Bbbk[\eta]$. This will be the stepping stone for generalizing the idea of mixed complexes to the derived setting.
\begin{parag}
	Let $\BGa$ the classifying stack for the affine smooth group scheme $\AAff^1\coloneqq\Spec(\Bbbk[t])$: it is an affine group stack which, in any characteristic, is equivalent to the spectrum of the derived commutative ring $\Sym_{\Bbbk}(\Bbbk[-1])$. However, when $\Bbbk$ is a base ring which contains $\QQ$, it is well-known that $$\Sym_{\Bbbk}(\Bbbk[-1])\simeq \Bbbk\oplus \Bbbk[-1]\eqqcolon \Bbbk[\varepsilon]$$with its square-zero extension commutative algebra structure. In this case, we can describe the semi-direct product $\BGagm$ of the affine group stacks $\BGa$ and $\Gm$ as the affine group stack whose algebra of functions is equivalent, as a commutative $\Bbbk$-algebra, to the formal $\Bbbk$-algebra $$\Bbbk\left[t,\hsp t^{-1}\right]\otimes_{\Bbbk}\lp\Bbbk\oplus\Bbbk\left[-1\right]\rp\simeq \Bbbk\left[t,\hsp t^{-1}\right]\oplus\Bbbk\left[t,\hsp t^{-1}\right][-1].$$Denoting again the generator in degree $-1$ with $\varepsilon$, the comultiplication for its cocommutative Hopf structure is given by the assignations $t\mapsto t\otimes t$ and $\varepsilon\mapsto t\otimes \varepsilon$.
\end{parag}
\begin{defn}
	\label{def:mixedgraded}
	The \textit{$\infinity$-category of mixed graded $\Bbbk$-modules}$$\Modmixgr_{\Bbbk}\coloneqq\Rep_{\BGagm}\simeq {\Qcoh}{\lp\mathsf{B}{\lp\BGagm\rp}\rp}$$ is the $\infinity$-category of representations of the derived group stack $\BGagm$. Equivalently, it is the $\infinity$-category of comodule objects for the Hopf algebra $\mathscr{O}_{\BGagm}$ in $\Mod_{\Bbbk}$.
\end{defn}
\begin{remark}
	\label{remark:mixedgradedexplicit}
	In the setting of commutative differential graded $\Bbbk$-algebras and chain complexes, a mixed graded $\Bbbk$-module can be thought as a chain complex of $\Bbbk$-modules $M_{\bullet}$, equipped with a decomposition of chain complexes of $\Bbbk$-modules $\left\{ \lp M_{\bullet}\rp_p\right\}_{p\in\ZZ}$ and with a morphism of chain complexes $$\varepsilon_p\colon \lp M_{\bullet}\rp_p\longrightarrow \lp M_{\bullet}\rp_{p-1}[-1]$$such that $\varepsilon_{p-1}[-1]\circ \varepsilon_p = 0$ for all $p\in\ZZ$.  The chain complex $\lp M_{\bullet}\rp_p$ is the \textit{$p$ weight component} of the mixed graded $\Bbbk$-module $M_{\bullet}$, while the morphism $\varepsilon$ is the \textit{mixed differential}. In the description given in Definition \ref{def:mixedgraded}, the action of $\Gm$ yields the weight grading, while the action of $\BGa$ yields the mixed differential; the fact that we are considering the semi-direct product assures us that the mixed differential decreases the weight grading by $-1$, i.e. the two actions are intertwined. See also \cite[Remark $1.1$]{PTVV}.
\end{remark}
\begin{remark}
	One could notice that in the explicit models provided by \cref{remark:mixedgradedexplicit} the internal differential and the mixed differential \textit{commute}, while in the classical notion of mixed complexes of \cref{def:mixedcomplexesclassical} they were required to \textit{anti-commute}. But since we are working with bi-graded objects, it is a standard computation to show that the two formalisms are completely equivalent, up to suitably changing the signs of the mixed differential.
\end{remark}
\begin{parag}
	\label{parag:monoidalstructureonmodmixgr}
	The $\infinity$-category of mixed graded $\Bbbk$-modules, being the \infinity-category of quasi-coherent sheaves over a derived stack, is naturally stable and it is endowed with a symmetric closed monoidal structure. Viewing $\mathscr{O}_{\BGagm}$-comodules as graded $\Bbbk$-modules endowed with a mixed differential, we can describe the internal tensor product, the internal mapping space and the unit for such monoidal structure as follows.\begin{enumerate}
		\item Given two mixed graded $\Bbbk$-modules $M_{\bullet}$ and $N_{\bullet}$, the tensor product $M_{\bullet}\otimesmixgr_{\Bbbk}N_{\bullet}$ is the mixed graded $\Bbbk$-module whose $p$-th weight component is given by the formula$$\lp M_{\bullet}\otimesmixgr_{\Bbbk} N_{\bullet}\rp_p\coloneqq\bigoplus_{i+j=p} M_i\otimes_{\Bbbk}N_j$$with mixed differential defined on every summand by the formula$$\varepsilon_M\otimes\operatorname{id}_N+\operatorname{id}_M\otimes\varepsilon_N\colon M_i\otimes N_j\longrightarrow\lp M_{i-1}\otimes N_j\rp \bigoplus \lp M_i\otimes N_{j-1}\rp[-1].$$
		\item The unit for $\otimesmixgr_{\Bbbk}$ is the mixed graded $\Bbbk$-module $\Bbbk(0)$, consisting of $\Bbbk$ sitting in pure weight $0$ with trivial mixed structure.
		\item Given two mixed graded $\Bbbk$-modules $M_{\bullet}$ and $N_{\bullet}$, the internal mapping space ${\Mapmixgr_{\Bbbk}}{\lp M_{\bullet},\hsp N_{\bullet}\rp}$ is the mixed graded $\Bbbk$-module whose $p$-th weight component is given by the formula$$\lp{\Mapmixgr_{\Bbbk}}{\lp M_{\bullet},\hsp N_{\bullet}\rp}\rp_p\coloneqq\prod_{q\in\ZZ}{{\Map_{\Mod_{\Bbbk}}}{\lp M_{q},\hsp N_{q+p}\rp}}$$with mixed differential$$\varepsilon_p\colon {\Mapmixgr_{\Bbbk}}{\lp M_{\bullet},\hsp N_{\bullet}\rp}_p\longrightarrow {\Mapmixgr_{\Bbbk}}{\lp M_{\bullet},\hsp N_{\bullet}\rp}_{p-1}[-1]$$given by the morphism whose $r$-th component is the sum of the morphism
		\begin{displaymath}
		\begin{tikzpicture}
		\node (a) at (-5,-0.2){$\underset{q\in\ZZ}{\prod}{\Map_{\Mod_{\Bbbk}}}{\lp M_{q},\hsp N_{q+p}\rp}$};
		\node (b) at (0,0){${\Map_{\Mod_{\Bbbk}}}{\lp M_{r},\hsp N_{r+p}\rp}$};
		\node (c) at (0,-1.5){${\Map_{\Mod_{\Bbbk}}}{\lp M_{r},\hsp N_{r+p-1}[-1]\rp}$};
		\node (d) at (4.9,-1.5){$\simeq {\Map_{\Mod_{\Bbbk}}}{\lp M_{r},\hsp N_{r+p-1}\rp}[-1]$};
		\draw[->,font=\scriptsize] (-2.9,0) to node [above]{$\pi_r$} (b);
		\draw[->,font=\scriptsize] (b) to node [right]{$\varepsilon_{r+p}^{N}\circ-$} (c);
		\end{tikzpicture}
		\end{displaymath}with the morphism\begin{displaymath}
		\begin{tikzpicture}
		\node (a) at (-3.5,-0.2){$\underset{q\in\ZZ}{\prod}{\Map_{\Mod_{\Bbbk}}}{\lp M_{q},\hsp N_{q+p}\rp}$};
		\node (b) at (1.5,0){${\Map_{\Mod_{\Bbbk}}}{\lp M_{r-1},\hsp N_{r-1+p}\rp}$};
		\node (c) at (1.5,-1.5){${\Map_{\Mod_{\Bbbk}}}{\lp M_{r}[1],\hsp N_{r-1+p}\rp}$};
		\node (d) at (6.3,-1.5){$\simeq {\Map_{\Mod_{\Bbbk}}}{\lp M_{r},\hsp N_{r+p-1}\rp}[-1].$};
		\draw[->,font=\scriptsize] (-1.4,0) to node [above]{$\pi_{r}$} (b);
		\draw[->,font=\scriptsize] (b) to node [right]{$-\circ\varepsilon_{r}^{M}[1]$} (c);
		\end{tikzpicture}
		\end{displaymath}
	\end{enumerate}
	The enrichment of $\Modmixgr_{\Bbbk}$ over $\Mod_{\Bbbk}$ is then given by$$\Map_{\Modmixgr_{\Bbbk}}{\lp M_{\bullet},\hsp N_{\bullet}\rp}\coloneqq\fib\lp\Mapmixgr_{\Bbbk}{\lp M_{\bullet},\hsp N_{\bullet}\rp}_0\overset{\varepsilon_0}{\longrightarrow}\Mapmixgr_{\Bbbk}{\lp M_{\bullet},\hsp N_{\bullet}\rp}_{-1}[-1]\rp.$$See also \cite[Section $1.1$]{CPTVV}.
\end{parag}
For future reference, we provide also a simple result about fully dualizable objects (in the sense of \cite[Section $4.6.1$]{ha}).
\begin{notation}
	\label{notation:mixedperfbounded}
	In the following, we shall denote by $\varepsilon\operatorname{-}\Perf^{{\operatorname{gr},-}}_{\Bbbk}$ the full sub-\infinity-category of $\Modmixgr_{\Bbbk}$ spanned by those mixed graded $\Bbbk$-modules $M_{\bullet}$ which are perfect in each weight and such that $M_p\simeq 0$ for all $p\gg 0$. Dually, we shall denote by $\varepsilon\operatorname{-}\Perf^{{\operatorname{gr},+}}_{\Bbbk}$ the full sub-\infinity-category of $\Modmixgr_{\Bbbk}$ spanned by those mixed graded $\Bbbk$-modules $M_{\bullet}$ which are perfect in each weight and such that $M_p\simeq 0$ for all $p\ll 0$. In a similar fashion, we shall denote by $\varepsilon\text{-}\Perf_{\Bbbk}^{\gr,\geqslant p}$ (respectively, $\varepsilon\text{-}\Perf_{\Bbbk}^{\gr,\leqslant q}$) the full sub-\infinity-category of $\Modmixgr_{\Bbbk}$ spanned by those mixed graded $\Bbbk$-modules $M_{\bullet}$ which are perfect in each weight and such that $M_n\simeq 0$ for all $n<p$ (respectively, for all $n>q$). Let us remark that we have inclusions of $\infinity$-categories$$\varepsilon\text{-}\Perf_{\Bbbk}^{\gr,\geqslant p}\subseteq \varepsilon\operatorname{-}\Perf^{{\operatorname{gr},+}}_{\Bbbk}$$and$$\varepsilon\text{-}\Perf_{\Bbbk}^{\gr,\leqslant q}\subseteq \varepsilon\operatorname{-}\Perf^{{\operatorname{gr},-}}_{\Bbbk}$$for all integers $p$ and $q$.
\end{notation}
\begin{propositionn}
	\label{prop:dualizableobjects}
	The full sub-\infinity-category of fully dualizable objects of mixed graded $\Bbbk$-modules coincides with the \infinity-category$${\lp\varepsilon\operatorname{-}\Perf^{{\operatorname{gr},+}}_{\Bbbk}\rp}\bigcap{\lp\varepsilon\operatorname{-}\Perf^{{\operatorname{gr},-}}_{\Bbbk}\rp}.$$
\end{propositionn}
\begin{proof}
	Let us denote by $M^{\vee}_{\bullet}$ the mixed graded $\Bbbk$-linear dual $\Mapmixgr_{\Bbbk}{\lp M_{\bullet},\hsp\Bbbk(0)\rp}$ of $M_{\bullet}$. We have an obvious evaluation morphism $M^{\vee}_{\bullet}\otimesmixgr_{\Bbbk} M_{\bullet}\longrightarrow\Bbbk(0)$ given  by the adjoint to the identity of $M^{\vee}_{\bullet}$. In virtue of \cite[Lemma $4.6.1.6$]{ha}, extending such evaluation morphism to a datum of full dualizability is equivalent to showing that tensoring with the evaluation produces an equivalence of $\Bbbk$-modules$$\Map_{\Modmixgr_{\Bbbk}}{\lp P_{\bullet},\hsp M^{\vee}_{\bullet}\otimesmixgr_{\Bbbk}N_{\bullet}\rp}\overset{\simeq}{\longrightarrow}\Map_{\Mod_{\Bbbk}}{\lp P_{\bullet}\otimesmixgr_{\Bbbk}M_{\bullet},\hsp N_{\bullet}\rp}$$for any mixed graded $\Bbbk$-modules $N_{\bullet}$ and $P_{\bullet}.$ This will be a consequence of the following Lemma.
	\begin{lemman}
		\label{lemma:dualizableobjects}
		For any $M_{\bullet}$ in ${\lp\varepsilon\operatorname{-}\Perf^{{\operatorname{gr},+}}_{\Bbbk}\rp}\cap{\lp\varepsilon\operatorname{-}\Perf^{{\operatorname{gr},-}}_{\Bbbk}\rp}$ and for an arbitrary mixed graded $\Bbbk$-module $N_{\bullet}$, the natural map of mixed graded $\Bbbk$-modules$$M^{\vee}_{\bullet}\otimesmixgr_{\Bbbk}N_{\bullet}\longrightarrow\Mapmixgr_{\Bbbk}{\lp M_{\bullet},\hsp N_{\bullet}\rp}$$is an equivalence.
	\end{lemman}
	\begin{proof}Let us recall how this map is defined. We have a chain of equivalences of mapping $\Bbbk$-modules\begin{align*}
		\Map_{\Modmixgr_{\Bbbk}}{\lp M_{\bullet}^{\vee}\otimesmixgr_{\Bbbk}N_{\bullet},\hsp \Mapmixgr_{\Bbbk}{\lp M_{\bullet},\hsp N_{\bullet}\rp}\rp}&\simeq \Map_{\Modmixgr_{\Bbbk}}{\lp M_{\bullet}^{\vee}\otimesmixgr_{\Bbbk}N_{\bullet}\otimesmixgr_{\Bbbk} M_{\bullet},\hsp N_{\bullet}\rp}\\&\simeq\Map_{\Modmixgr_{\Bbbk}}{\lp M_{\bullet}^{\vee}\otimesmixgr_{\Bbbk}M_{\bullet}\otimesmixgr_{\Bbbk} N_{\bullet},\hsp N_{\bullet}\rp}
		\end{align*}
		and so, tensoring the evaluation map $M_{\bullet}^{\vee}\otimesmixgr_{\Bbbk}M_{\bullet}\longrightarrow\Bbbk(0)$ with the identity of $N_{\bullet}$ one has the desired map $M^{\vee}_{\bullet}\otimesmixgr_{\Bbbk}N_{\bullet}\longrightarrow\Mapmixgr_{\Bbbk}{\lp M_{\bullet},\hsp N_{\bullet}\rp}.$ To prove it is an equivalence, we check it weight-wise: the left hand side is described in weight $p$ by the formula$$\lp M^{\vee}_{\bullet}\otimesmixgr_{\Bbbk}N_{\bullet}\rp_p\simeq \bigoplus_{i+j=p}\Map_{\Mod_{\Bbbk}}{\lp M_{-i},\hsp \Bbbk\rp}\otimes_{\Bbbk}N_j$$whereas the right hand side is described in weight $p$ by the formula$$\Mapmixgr_{\Bbbk}{\lp M_{\bullet},\hsp N_{\bullet}\rp}_p\simeq \prod_{h\in\ZZ}\Map_{\Mod_{\Bbbk}}{\lp M_h,\hsp N_{h+p}\rp}.$$With a change of indices $i\coloneqq -h$ (hence $h+p=p-i=j$), since $M_{-i}$ is perfect in $\Mod_{\Bbbk}$ for any integer $i$, we can rewrite the latter as$$\prod_{i\in\ZZ}\Map_{\Mod_{\Bbbk}}{\lp M_{-i},\hsp N_j\rp}\simeq \prod_{i\in\ZZ}\Map_{\Mod_{\Bbbk}}{\lp M_{-i},\hsp \Bbbk\rp}\otimes_{\Bbbk}N_j.$$Finally, since the grading of $M_{\bullet}$ is bounded above and below, we know that $\Map_{\Mod_{\Bbbk}}{\lp M_{-i},\hsp \Bbbk\rp}$ can be non-zero only for finitely many indices, hence the product is actually a direct sum. Therefore, the map is an equivalence. 
	\end{proof}
	\cref{lemma:dualizableobjects} shows that $M^{\vee}_{\bullet}\otimesmixgr_{\Bbbk}N_{\bullet}\simeq\Mapmixgr_{\Bbbk}{\lp M_{\bullet},\hsp N_{\bullet}\rp}$ for any mixed graded $\Bbbk$-module $N_{\bullet}$ and for any mixed graded $\Bbbk$-module $M_{\bullet}$ which is bounded and perfect in each weight, and this equivalence is provided exactly by tensoring the adjoint map to the evaluation $M^{\vee}_{\bullet}\otimesmixgr_{\Bbbk}M_{\bullet}\to\Bbbk(0)$ with the identity of $N_{\bullet}$. This proves the second assertion of \cite[Lemma $4.6.1.6$]{ha}, hence the dualizability of $M_{\bullet}$.\\
	
	Proving that any fully dualizable mixed graded $\Bbbk$-module lies in ${\lp\varepsilon\operatorname{-}\Perf^{{\operatorname{gr},+}}_{\Bbbk}\rp}\cap{\lp\varepsilon\operatorname{-}\Perf^{{\operatorname{gr},-}}_{\Bbbk}\rp}$ relies on the description of dualizable graded $\Bbbk$-modules. Indeed, we have a forgetful \infinity-functor$$\oblv_{\varepsilon}\colon\Modmixgr_{\Bbbk}\longrightarrow\Modgr_{\Bbbk},$$	described (model independently) in the following way. Let us interpret mixed graded ${\Bbbk}$-modules as quasi-coherent $\mathscr{O}$-sheaves on ${\mathsf{B}}{\lp\BGagm\rp}$. We have a right split extension of group stacks
	\begin{align}
	\label{ses:groupstacks}
	\BGa\longrightarrow\BGagm\longrightarrow\Gm.
	\end{align}
	The splitting morphism $\Gm\to\BGagm$ induces in this way a morphism $\BGm\to{\mathsf{B}}{\lp\BGagm\rp}$, which induces a pullback $\infinity$-functor
	\begin{align}
	\label{functor:forgetful}
	\oblv_{\varepsilon}\colon{\Qcoh}{\lp{\mathsf{B}}{\lp\BGagm\rp}\rp}\longrightarrow{\Qcoh}{\lp\BGm\rp}.
	\end{align}
	By the known equivalence between $\Gm$-equivariant quasi-coherent sheaves on $\Spec({\Bbbk})$ and graded ${\Bbbk}$-modules, this reduces to the forgetful $\infinity$-functor that sends a graded mixed ${\Bbbk}$-module $M_{\bullet}$ to the underlying graded ${\Bbbk}$-module $M_{\bullet}$ by forgetting the mixed structure. In particular, the forgetful \infinity-functor $\oblv_{\varepsilon}\colon\Modmixgr_{\Bbbk}\to\Modgr_{\Bbbk}$, being a quasi-coherent pullback \infinity-functor, preserves both tensor products and internal mapping spaces. Hence, any dualizable object in $\Modmixgr_{\Bbbk}$, after forgetting the mixed structure, must become dualizable also in $\Modgr_{\Bbbk}$. So we just need to prove that any dualizable graded $\Bbbk$-module $M_{\bullet}$ is perfect in each weight and endowed with bounded grading. Indeed a map $\Bbbk(0)\to M_{\bullet}\otimes^{\gr}_{\Bbbk}M^{\vee}_{\bullet}$ corresponds to an element in the weight $0$ part of the $\Bbbk$-module$$\lp M_{\bullet}\otimes^{\gr}_{\Bbbk}M^{\vee}_{\bullet}\rp_0\simeq \bigoplus_{n\in\ZZ}M_n\otimes_{\Bbbk}\Map_{\Mod_{\Bbbk}}{\lp M_{n},\hsp\Bbbk\rp}.$$In particular, a coevaluation morphism for $M_{\bullet}$ in $\Modgr_{\Bbbk}$ provides, after post-composing with the natural map$$\bigoplus_{n\in\ZZ}M_n\otimes_{\Bbbk}\Map_{\Mod_{\Bbbk}}{\lp M_{n},\hsp\Bbbk\rp}\longrightarrow\prod_{n\in\ZZ}M_n\otimes_{\Bbbk}\Map_{\Mod_{\Bbbk}}{\lp M_n,\hsp\Bbbk\rp}$$and then projecting on the $p$-th direct summand, a coevaluation $\Bbbk\to M_p\otimes_{\Bbbk}\Map_{\Mod_{\Bbbk}}{\lp M_p,\hsp\Bbbk\rp}$ for any integer $p$. The fact that all the coherences of the definition of a coevaluation morphism are satisfied is a consequence of the fact that $\Bbbk(0)\to M_{\bullet}\otimes_{\Bbbk}^{\gr}M^{\vee}_{\bullet}$ is assumed to be a coevaluation itself: one can easily see it by checking the coherence diagram for the component in weight $-p$ of$$\operatorname{id}_{M^{\vee}_{\bullet}}\colon M^{\vee}_{\bullet}\xrightarrow{\operatorname{id}\otimes\operatorname{coev}} M^{\vee}_{\bullet}\otimesmixgr_{\Bbbk}M_{\bullet}\otimesmixgr_{\Bbbk}M^{\vee}_{\bullet}\xrightarrow{\operatorname{ev}\otimes\operatorname{id}} M^{\vee}_{\bullet}.$$By the characterization of dualizable objects in $\Mod_{\Bbbk}$, this shows that $M_p$ must be perfect for any integer $p$.\\
	For the statement about the upper and lower bound of the grading, let us remark that an analogous argument to the one in the proof of \cref{lemma:dualizableobjects} provides always a map$$\Mapin^{\gr}_{\Bbbk}{\lp M_{\bullet},\hsp\Bbbk(0)\rp}\otimes^{\gr}_{\Bbbk}N_{\bullet}\longrightarrow\Mapin^{\gr}_{\Bbbk}{\lp M_{\bullet},\hsp N_{\bullet}\rp}$$where $\Mapin^{\gr}_{\Bbbk}$ is the internal graded mapping $\Bbbk$-module \infinity-functor for the closed symmetric monoidal \infinity-category $\Modgr_{\Bbbk}$. If $M_{\bullet}$ is not bounded in both directions, in general it can fail to be an equivalence. For example, if $M_{p}\not\simeq0$ for all non-negative integers $p$, considering $N_{\bullet}\coloneqq M_{\bullet}$, then the weight $0$ component of the above map is described by$$\bigoplus_{n\in\ZZ}\Map_{\Mod_{\Bbbk}}{\lp M_{n},\hsp \Bbbk\rp}\otimes_{\Bbbk}M_n\longrightarrow\prod_{n\in\ZZ} \Map_{\Mod_{\Bbbk}}{\lp M_n,\hsp M_n\rp}$$which of course can never be an equivalence without the boundness assumption.
\end{proof}
\begin{remark}
	The $\infinity$-category $\Modmixgr_{\Bbbk}$ is equivalent, as a stable symmetric monoidal \infinity-category, to the $\infinity$-category of comodules on $\mathscr{O}_{\BGagm}$. We  can alternatively consider the $\infinity$-category of comodules on $\mathscr{O}_{\Omega_0\Ga\rtimes \Gm}$, that is the $\infinity$-category of comodules on the algebra of functions on the affine group stack $$\Omega_0\Ga\rtimes \Gm\simeq {\Spec}{\lp\Bbbk\left[t,\hsp t^{-1}\right]\otimes_{\Bbbk}\lp\Bbbk\oplus\Bbbk[1]\rp\rp}.$$The two theories are equivalent: in the latter case, the mixed differential $\varepsilon$ is a morphism of degree $-1$ instead of degree $1$ (i.e., the mixed structure is the datum of a map $\varepsilon_p\colon M_p \to M_{p-1}[1]$). An explicit equivalence between the two comodule theories simply sends a comodule $M_{\bullet}$ over $\mathscr{O}_{\BGagm}$ to the comodule over $\mathscr{O}_{\Omega_0\Ga\rtimes \Gm}$ given in weight $p$ by the $\Bbbk$-module $M_p[-2p]$. See also \cite[Remark $1.1.3$]{CPTVV}.
\end{remark}
\begin{parag}
	\label{parag:modeladjunction}
	Let ${\Bbbk}(p)$ denote the mixed graded ${\Bbbk}$-module consisting of ${\Bbbk}$ sitting in pure weight $p$ and homological degree $0$. We have an adjunction
	\begin{align}
	\label{adjunction:realizationadjunction}
	\begin{tikzpicture}[scale=0.75,baseline=-0.5ex]
	\node at (-3.5,0){$-\otimes_{\Bbbk} \Bbbk(0)\colon$};
	\node (a) at (-1.5,0){$\Mod_{\Bbbk}$};
	\node (b) at (1.5,0){$\Modmixgr_{\Bbbk}$};
	\node at (3.1,-0){$: \left|-\right|$};
	\draw[->] ([yshift=3.5pt]a.east) -- ([yshift=3.5pt]b.west);
	\draw[->] ([yshift=-3.5pt]b.west) -- ([yshift=-3.5pt]a.east);
	\end{tikzpicture}
	\end{align}
	where the left adjoint simply sends a $\Bbbk$-module $M$ to the mixed graded $\Bbbk$-module consisting of $M$ concentrated in weight $0$, and the right adjoint $\left|-\right|$ is the $\infinity$-functor that sends a mixed graded ${\Bbbk}$-module $M_{\bullet}$ to the mapping $\Bbbk$-module$$\left|M_{\bullet}\right|\coloneqq\Map_{\Modmixgr_{\Bbbk}}{\lp {\Bbbk}(0),\hsp M_{\bullet}\rp}.$$This right adjoint is called the \textit{realization $\infinity$-functor}: a strict model of $\left|M_{\bullet}\right|$ is provided by the chain complex of ${\Bbbk}$-modules $$\prod_{p\geqslant 0}M_{-p}[-2p]$$endowed with the total differential, sum of the usual differential of chain complexes and the mixed differential (\cite[Proposition $1.5.1$]{CPTVV}).
\end{parag}
\begin{notation}
	In the remainder of this work, the $\infinity$-functor $-\otimes_{\Bbbk}\Bbbk(q)\colon\Mod_{\Bbbk}\to\Modmixgr_{\Bbbk}$ which sends a $\Bbbk$-module $M$ to the mixed graded $\Bbbk$-module $M$ concentrated in weight $q$ with trivial mixed differential shall be denoted simply as $(-)(q)$.
\end{notation}
\begin{construction}
	For our purposes, it will be convenient to introduce another realization $\infinity$-functor that can keep track of the ${\Bbbk}$-modules in positive weights of a mixed graded module $M_{\bullet}$. Let us recall (\cite[Section $1.5$]{CPTVV}) that for all $p\in\ZZ$ we have that
	\begin{align}
	\label{map}
	{\Map_{\Modmixgr_{\Bbbk}}}{\lp {\Bbbk}(i)[-2i],\hsp {\Bbbk}(i-1)[-2(i-1)]\rp}\simeq {\Bbbk}.
	\end{align}
	So we have a pro-object in $\Modmixgr_{\Bbbk}$, defined by
	\begin{align}
	\label{taterealizationexplicit}
	{\Bbbk}(\infinity) \coloneqq\left\{\ldots\to {\Bbbk}(i)[-2i]\to {\Bbbk}(i-1)[-2(i-1)]\to\ldots \to {\Bbbk}(1)[-2]\to {\Bbbk}(0)\right\}
	\end{align}
	where the morphism ${\Bbbk}(i)[-2i]\to {\Bbbk}(i-1)[-2(i-1)]$ is the unique morphism corresponding to the unit $1$ of ${\Bbbk}$ under the equivalence \ref{map}.
\end{construction}
\begin{defn}[{\cite[Definition $1.5.2$]{CPTVV}}]
	\label{def:taterealization}
	The \textit{Tate} or \textit{stabilized realization} $\infinity$-functor is defined as
	\begin{displaymath}
	\left|-\right|^{\operatorname{t}}\coloneqq{\Map_{\Modmixgr_{\Bbbk}}}{\lp {\Bbbk}(\infinity),\hsp -\rp}\colon \Modmixgr_{\Bbbk}\longrightarrow \Ind\lp\Mod_{\Bbbk}\rp\overset{\colim}{\longrightarrow}\Mod_{\Bbbk}.
	\end{displaymath}
\end{defn}
\begin{parag}
	\label{parag:tateexplicit}
	Again, working with explicit models given by graded chain complexes and mixed differentials, the $\infinity$-functor of Definition \ref{def:taterealization} sends a mixed graded ${\Bbbk}$-module $M_{\bullet}=\{M_p\}_p$ to the ${\Bbbk}$-module $$\left|
	M_{\bullet}\right|^{\operatorname{t}}\coloneqq\underset{i\leqslant 0}{\colim}\prod_{p\geqslant i}M_{-p}[-2p]$$again endowed with the total differential. There is a natural transformation of $\infinity$-functors$$\left|-\right|\Rightarrow\left|-\right|^{\operatorname{t}}\colon\Modmixgr_{\Bbbk}\longrightarrow\Mod_{\Bbbk}$$which is induced by the map of pro-objects $\Bbbk(\infinity)\to\Bbbk(0)$ (the latter seen as a constant pro-object). Working with explicit models, we easily see that the natural transformation above is described, on a given mixed graded $\Bbbk$-module $M_{\bullet}$, as the inclusion$$\prod_{p\geqslant 0}M_{-p}[-2p]\longhookrightarrow\underset{i\leqslant 0}{\colim}\prod_{p\geqslant i}M_{-p}[-2p].$$In particular, the map above is an equivalence whenever $M_{\bullet}$ is trivial in all positive weights. 
\end{parag}
\begin{construction}
	For all $p\in\ZZ$ we have a \textit{$p$ weight part $\infinity$-functor}
	\begin{displaymath}
	(-)_p\colon \Modmixgr_{\Bbbk}\longrightarrow\Mod_{\Bbbk}
	\end{displaymath}
	described (model independently) in the following way. Let us recall the forgetful \infinity-functor $\oblv_{\varepsilon}\colon\Modmixgr_{\Bbbk}\to\Modgr_{\Bbbk}$ described in \ref{functor:forgetful} as the quasi-coherent pullback \infinity-functor induced by atlas $\mathsf{B}\Gm\to{\mathsf{B}{\lp\BGagm\rp}}$. Since we have an equivalence of $\infinity$-categories $$\Modgr_{\Bbbk}\simeq \prod_p\Mod_{\Bbbk},$$one can project onto the $p$-th coordinate: this composition yields the desired $\infinity$-functor.\\
	The $\infinity$-functor $(-)_0$ has a left adjoint \begin{align}
	\label{functor:0weightadjoint}
	\operatorname{Free}_{\varepsilon}\colon \Mod_{\Bbbk}\longrightarrow\Modmixgr_{\Bbbk}
	\end{align} (the free mixed graded ${\Bbbk}$-module construction, see \cite[$1.4.1$]{CPTVV}). 
\end{construction}
\begin{parag}
	\label{parag:freeleftadjointmixedgraded}
	By pre-composing $(-)_0$ with the weight shift by $q$ on the left $\infinity$-endofunctor
	\begin{displaymath}
	(-)\lpd q\rpd \coloneqq-\otimesmixgr_{\Bbbk} {\Bbbk}(q)\colon \Modmixgr_{\Bbbk}\longrightarrow\Modmixgr_{\Bbbk}
	\end{displaymath}
	which informally sends a mixed graded ${\Bbbk}$-module $M_{\bullet}=\left\{M_p\right\}_{\!p\in\ZZ}$ to the mixed graded ${\Bbbk}$-module $M\lpd q\rpd_{\bullet}\coloneqq\left\{M_{p-q}\right\}_{\!p\in\ZZ}$, and by post-composing $\operatorname{Free}_{\varepsilon}$ with the shift $\infinity$-functor $\lpd q\rpd$ we obtain for all $q\in\ZZ$ another adjunction 
	\begin{align}
	\label{adjunction:weightadjunction}
	\begin{tikzpicture}[scale=0.75,baseline=0.5ex]
	\node at (-3.5,0){$\operatorname{Free}_{\varepsilon}\lpd q\rpd \colon$};
	\node (a) at (-1.5,0){$\Mod_{\Bbbk}$};
	\node (b) at (1.5,0){$\Modmixgr_{\Bbbk}$};
	\node at (3.4,-0){$\colon(-)_q$};
	\draw[->] ([yshift=3pt]a.east) -- ([yshift=3pt]b.west);
	\draw[->] ([yshift=-3pt]b.west) -- ([yshift=-3pt]a.east);
	\end{tikzpicture}
	\end{align}
	which generalizes the adjunction \ref{functor:0weightadjoint} to all weights. Let us remark that the description of this left adjoint is  very explicit: indeed, the proof of \cite[Proposition $1.3.8$]{CPTVV} together with the observations made in \cite[Section $1.4.1$]{CPTVV} shows that such left adjoint simply sends a $\Bbbk$-module $M$ to the mixed graded $\Bbbk$-module $\operatorname{Free}_{\varepsilon}(M)\lpd q\rpd$ consisting of $M$ in weight $q$, $M[1]$ in weight $q-1$, and with mixed structure given by the natural equivalence $M\simeq M[1][-1]$.
\end{parag}
Let us remark that, for any integer $q$, one has an adjoint pair	$$
\begin{tikzpicture}[scale=0.75,baseline=0.5ex]
\node at (-3.2,0){$(-)(q)\colon$};
\node (a) at (-1.5,0){$\Mod_{\Bbbk}$};
\node (b) at (1.5,0){$\Modgr_{\Bbbk}$};
\node at (3.1,-0){$\colon(-)_q$};
\draw[->] ([yshift=3pt]a.east) -- ([yshift=3pt]b.west);
\draw[->] ([yshift=-3pt]b.west) -- ([yshift=-3pt]a.east);
\end{tikzpicture}$$
which is an ambidextrous adjunction. The $\infinity$-functor $(-)_q\colon\Modmixgr_{\Bbbk}\to\Mod_{\Bbbk}$ is canonically equivalent to the composition $$\Modmixgr_{\Bbbk}\xrightarrow{\oblv_{\varepsilon}}\Modgr_{\Bbbk}\xrightarrow{(-)_q}\Mod_{\Bbbk},$$so the existence of the adjoint $\operatorname{Free}_{\varepsilon}\lpd q\rpd$ yields the following result.
\begin{lemman}
	\label{lemma:colimitsinmixgr}
	Limits and colimits in $\Modmixgr_{\Bbbk}$ are computed weight-wise. 
\end{lemman}
\begin{proof}
	Let us recall that the forgetful functor $\oblv_{\varepsilon}\colon \Modmixgr_{\Bbbk}\longrightarrow\Modgr_{\Bbbk}$ corresponds geometrically to the pullback functor \ref{functor:forgetful}. Every pullback $\infinity$-functor commutes with colimits (\cite[Proposition $2.7.17$]{dagviii}) and thus colimits in $\Modmixgr_{\Bbbk}$ are computed as in $\Modgr_{\Bbbk}$. But $$\Modgr_{\Bbbk}\simeq \prod_{q\in\ZZ}\Mod_{\Bbbk},$$and colimits on the right hand side are computed weight-wise.\\
	The assertion for limits is more subtle, since in general pullback $\infinity$-functors do not preserve them. But in this case, we can use the discussion of \ref{parag:freeleftadjointmixedgraded}: indeed, the composition$$(-)_p\colon\Modmixgr_{\Bbbk}\overset{\oblv_{\varepsilon}}{\longrightarrow}\Modgr_{\Bbbk}\longrightarrow\Mod_{\Bbbk}$$preserves limits for all $p$'s, since it admits an explicit left adjoint. The product$$\prod_p(-)_p\colon \Modmixgr_{\Bbbk}\overset{\oblv_{\varepsilon}}{\longrightarrow} \Modgr_{\Bbbk}\overset{\left\{(-)_p\right\}_p}{\longrightarrow}\prod_{p\in\ZZ}\Mod_{\Bbbk}\simeq\Modgr_{\Bbbk}$$is canonically equivalent to the forgetful \infinity-functor $\oblv_{\varepsilon}$, and it commutes with limits since it is the product of \infinity-functors which commute with limits. 
	%in fact, we can exhibit explicitly 
\end{proof}
\begin{porism}
	\label{porism:rightadjointforgetful}
	Since both $\Modmixgr_{\Bbbk}$ and $\Modgr_{\Bbbk}$ are presentable $\infinity$-categories, the forgetful $\infinity$-functor $\oblv_{\varepsilon}$ is both a left and right adjoint in virtue of the Adjoint Functor Theorem (\cite[Corollary $5.5.2.9$]{htt}). Such adjoints $\operatorname{L}_{\varepsilon}$ and $\operatorname{R}_{\varepsilon}$ can be described with explicit models in the following way.
	\begin{itemize}
		\item The left adjoint $\operatorname{L}_{\varepsilon}$ sends a graded $\Bbbk$-module $M_{\bullet}$ to the mixed graded $\Bbbk$-module $\operatorname{L}_{\varepsilon}M_{\bullet}$ defined in weight $p$ by the formula$$\lp\operatorname{L}_{\varepsilon}M_{\bullet}\rp_p\simeq M_p\oplus M_{p+1}[1],$$and whose $\BGa$-action is described by the morphism $$\varepsilon_p\colon M_p\oplus M_{p+1}[1]\longrightarrow\lp M_{p-1}\oplus M_{p-1+1} [1]\rp [-1]\simeq M_{p-1}[-1]\oplus M_p$$given by the canonical equivalence $M_p\simeq M_p[1][-1]$ and the zero map on $M_{p+1}[1]$.
		\item The right adjoint $\operatorname{R}_{\varepsilon}$ sends a graded $\Bbbk$-module $M_{\bullet}$ to the mixed graded $\Bbbk$-module $\operatorname{R}_{\varepsilon} M_{\bullet}$ defined in weight $p$ by the formula $\lp\operatorname{R}_{\varepsilon}M_{\bullet}\rp_p\simeq M_p\oplus M_{p-1}[-1]$, and whose $\BGa$-action is described by the morphism
		\begin{displaymath}
		\varepsilon_p\colon M_p\oplus M_{p-1}[-1]\longrightarrow \lp M_{p-1}\oplus M_{p-2}[-1]\rp [-1]\simeq M_{p-1}[-1]\oplus M_{p-2}[-2]
		\end{displaymath}
		given by the canonical equivalence $M_{p-1}[-1]\simeq M_{p-1}[-1]$ and the zero map on $M_p$.
	\end{itemize}
	It is now clear that the $\infinity$-functor $\operatorname{Free}_{\varepsilon}\lpd q\rpd$ of \ref{adjunction:weightadjunction} is canonical equivalent to the composition $$\operatorname{L}_{\varepsilon}\circ\hsp(-)(q)\colon\Mod_{\Bbbk}\longrightarrow\Modmixgr_{\Bbbk}.$$
\end{porism}
\begin{construction}
	\label{construction:trivialmixedstructure}
	We shall provide one last,  but very useful, construction that highlights how mixed graded $\Bbbk$-modules are to be thought as graded $\Bbbk$-modules with extra structure. Recall the right split extension of group stacks \ref{ses:groupstacks}. The morphism $\BGagm\to\Gm$ induces a pullback $\infinity$-functor at the level of quasi-coherent sheaves$$\triv_{\varepsilon}\colon{\Qcoh}{\lp\mathsf{B}\Gm\rp}\longrightarrow{\Qcoh}{\lp\mathsf{B}{\lp\BGagm\rp}\rp}.$$By functoriality, one has a natural equivalence
	\begin{align}
	\label{equivalence:oblvtriv}
	\operatorname{id}_{\Modgr_{\Bbbk}}\simeq\oblv_{\varepsilon}\circ\triv_{\varepsilon},
	\end{align}
	i.e., this $\infinity$-functor is a section of the forgetful $\infinity$-functor $\oblv_{\varepsilon}\colon\Modmixgr_{\Bbbk}\to\Modgr_{\Bbbk}$. In general, $\triv_{\varepsilon}$ has to be understood as the $\infinity$-functor that endows a graded $\Bbbk$-module with the trivial mixed structure, i.e., the zero morphism between each weight.
	\begin{warning}
		The $\infinity$-functor $\triv_{\varepsilon}\colon \Modgr_{\Bbbk}\to\Modmixgr_{\Bbbk}$ commutes with all limits and colimits since the forgetful $\infinity$-functor $\oblv_{\varepsilon}$ creates them, and $\oblv_{\varepsilon}\circ\triv_{\varepsilon}\simeq \operatorname{id}_{\Modgr_{\Bbbk}}$. Moreover, it is a straight-forward consequence of the discussion in \ref{parag:monoidalstructureonmodmixgr} that $\triv_{\varepsilon}$ is also strongly monoidal. It is, however, \textit{not} fully faithful: such pullback \infinity-functor is in fact the left adjoint to the push-forward \infinity-functor $\Modmixgr_{\Bbbk}\to\Modgr_{\Bbbk}$ which agrees with the \textit{weighted negative cyclic \infinity-functor} $\operatorname{NC}^{\operatorname{w}}\colon\Modmixgr_{\Bbbk}\to\Modgr_{\Bbbk}$ of \cite[Remark $1.5$]{PTVV}. Unraveling all definitions, this \infinity-functor sends any mixed graded $\Bbbk$-module $M_{\bullet}$ to the graded $\Bbbk$-module $\operatorname{NC}^{\operatorname{w}}(M_{\bullet})$ described in each weight by the formula$$\operatorname{NC}^{\operatorname{w}}(M_{\bullet})_p\coloneqq \left|M_{\bullet}\lpd p\rpd\right|.$$In particular, the unit for the adjunction $\triv_{\varepsilon}\dashv\operatorname{NC}^{\operatorname{w}}$ cannot be an equivalence. For example, if we consider the graded $\Bbbk$-module consisting of $\Bbbk$ sitting in pure weight $0$, then the formula provided in \cite{PTVV} for the weighted negative cyclic \infinity-functor yields that$$\operatorname{NC}^{\operatorname{w}}(\triv_{\varepsilon}\Bbbk(0))_p\simeq \Bbbk[-2p]$$for all $p\geqslant 0.$ However, the \infinity-functor $\triv_{\varepsilon}$ is fully faithful on those graded $\Bbbk$-modules which are concentrated in a single weight: indeed, the right adjoint of the \infinity-functor\begin{align}
		\label{functor:insertioninweightq}
		(-)(q)\coloneqq \triv_{\varepsilon}\circ\hsp (-)(q)\colon\Mod_{\Bbbk}\longrightarrow\Modgr_{\Bbbk}
		\longrightarrow\Modmixgr_{\Bbbk}
		\end{align}is the composition of $\operatorname{NC}^{\operatorname{w}}$ with the projection on the $q$-th weight component. In particular, the unit map$$M\longrightarrow \lp\operatorname{NC}^{\operatorname{w}}(\triv_{\varepsilon}M(q))\rp_q$$for the adjunction $\triv_{\varepsilon}\circ\hsp (-)(q)\dashv (-)_q\circ \operatorname{NC}^{\operatorname{w}}$ is an equivalence, and since $(-)(q)$ is fully faithful for any integer $q$ we have a chain of equivalences of mapping spaces\begin{align*}
		\Map_{\Modgr_{\Bbbk}}\lp M(q),\hsp N(q)\rp &\simeq\Map_{\Mod_{\Bbbk}}\lp M,\hsp N\rp  \\&\simeq \Map_{\Modmixgr_{\Bbbk}}\lp \triv_{\varepsilon}(M(q)),\hsp \triv_{\varepsilon}(N(q))\rp.
		\end{align*}
	\end{warning}
\end{construction}
\subsection{Naive and clever truncations}For any integer $p$, we can consider the full sub-$\infinity$-category $\varepsilon\operatorname{-}\Mod^{{\operatorname{gr},\leqslant p}}_{\Bbbk}$ of $\Modmixgr_{\Bbbk}$ spanned by all those mixed graded $\Bbbk$-modules whose $q$ weight component is $0$ whenever $q>p$.\\The  inclusion $$\iota_{\leqslant p}\colon \varepsilon\operatorname{-}\Mod^{{\operatorname{gr},\leqslant p}}_{\Bbbk}\subseteq \Modmixgr_{\Bbbk}$$commutes with all limits and colimits, since they are computed weight-wise (\cref{lemma:colimitsinmixgr}).
It follows that for all integers $p$ the inclusion $\iota_{\leqslant p}$, being a limit and colimit preserving $\infinity$-functor between presentable $\infinity$-categories admits both a left and a right adjoint, again in virtue of the Adjoint Functor Theorem.
\begin{defn}\
	\label{def:truncationindegleqp}
	\begin{enumerate}[label=\arabic{enumi}., ref=\thelemman.\arabic{enumi}]
		\item \label{def:naivetruncationindegleqp}The $\infinity$-functor\begin{align*}
		%	\label{functor:naivetruncationmixgr}
		\sigma_{\leqslant p}\colon \Modmixgr_{\Bbbk}\longrightarrow\varepsilon\operatorname{-}\Mod^{{\operatorname{gr},\leqslant p}}_{\Bbbk}
		\end{align*}\textit{right} adjoint to the inclusion $$\iota_{\leqslant p}\colon \varepsilon\operatorname{-}\Mod^{{\operatorname{gr},\leqslant p}}_{\Bbbk}\subseteq \Modmixgr_{\Bbbk}$$is the \textit{naive truncation in weights $\leqslant p$.}
		\item \label{def:clevertruncationindegleqp}The $\infinity$-functor\begin{align*}
		%	\label{functor:naivetruncationmixgr}
		\theta_{\leqslant p}\colon \Modmixgr_{\Bbbk}\longrightarrow\varepsilon\operatorname{-}\Mod^{{\operatorname{gr},\leqslant p}}_{\Bbbk}
		\end{align*}\textit{left} adjoint to the inclusion $$\iota_{\leqslant p}\colon \varepsilon\operatorname{-}\Mod^{{\operatorname{gr},\leqslant p}}_{\Bbbk}\subseteq \Modmixgr_{\Bbbk}$$is the \textit{clever truncation in weights $\leqslant p$.}
	\end{enumerate}
\end{defn}
The $\infinity$-functor of Definition \ref{def:naivetruncationindegleqp} must be understood as the $\infinity$-functor which sends a mixed graded $\Bbbk$-module $M_{\bullet}$ to the mixed graded $\Bbbk$-module $\sigma_{\leqslant p}M_{\bullet}$ whose $q$ weight component is $M_p$ if $q\leqslant p$, and $0$ otherwise, with obviously induced mixed differential. This is different from the $\infinity$-functor of Definition \ref{def:clevertruncationindegleqp} which in turn is an analogue of the clever truncation of chain complexes in the context of mixed graded $\Bbbk$-modules.
% Indeed, $\lp\theta_{\leqslant p}M_{\bullet}\rp_q$ is equivalent to $M_q$ in weights $q<p$, it is trivial when $q>p$, and in weight $q=p$ replaces $M_p$ with the cofiber of the mixed differential $$M_{p+1}[1]\xrightarrow{\varepsilon_{p+1}[1]}M_p.$$
\begin{remark}
	\label{remark:inclusionstruncationsmixedgraded}
	By the adjunction $\iota_{\leqslant p}\dashv\sigma_{\leqslant p}$, for any mixed graded $\Bbbk$-module $M_{\bullet}$ we have a counit morphism$$\epsilon_{M_{\bullet}}\colon \iota_{\leqslant p} \lp\sigma_{\leqslant p} M_{\bullet}\rp\longrightarrow M_{\bullet}$$ which includes its part concentrated in weights $\leqslant p$ in itself.\\For any $q\leqslant p$, we have that $$\sigma_{\leqslant q}\lp \iota_{\leqslant p}\lp\sigma_{\leqslant p}M_{\bullet}\rp\rp\simeq\sigma_{\leqslant q}M_{\bullet}$$so the counit morphisms yield inclusions$$\epsilon_{\iota_{\leqslant p}\lp\theta_{\leqslant p}M_{\bullet}\rp}\colon  \iota_{\leqslant q} \lp\sigma_{\leqslant q}\lp \iota_{\leqslant p}\lp\sigma_{\leqslant p}M_{\bullet}\rp\rp\rp\simeq  \iota_{\leqslant q}\lp\sigma_{\leqslant q} M_{\bullet}\rp\longrightarrow \iota_{\leqslant p}\lp\sigma_{\leqslant p}M_{\bullet}\rp.$$\\
	We shall commit a slight abuse of notation, and identify $\sigma_{\leqslant q}M_{\bullet}$ with its inclusion $\iota_{\leqslant q}\lp\sigma_{\leqslant q}M_{\bullet}\rp$ in $\Modmixgr_{\Bbbk}$; in particular, we shall write the inclusion morphisms above simply as
	\begin{align}
	\sigma_{\leqslant q}M_{\bullet}\longhookrightarrow\sigma_{\leqslant p}M_{\bullet}\longhookrightarrow M_{\bullet}.
	\end{align}
	Using the dual properties of the adjunction $\theta_{\leqslant p}\dashv \iota_{\leqslant p}$, we easily obtain dual properties for the clever truncation \infinity-functor.
\end{remark}
\begin{parag}
	A completely dual discussion can be carried out by replacing $\varepsilon\operatorname{-}\Mod^{{\operatorname{gr},\leqslant p}}_{\Bbbk}$ with $\varepsilon\operatorname{-}\Mod^{{\operatorname{gr},\geqslant p}}_{\Bbbk}$, which is now the full sub-\infinity-category consisting of those mixed graded $\Bbbk$-modules which are trivial in weights $q<p$. Now, the inclusion$$\iota_{\geqslant p}\colon\varepsilon\operatorname{-}\Mod^{{\operatorname{gr},\geqslant p}}_{\Bbbk}\subseteq\Modmixgr_{\Bbbk}$$preserves all limits and colimits and hence admits both left and right adjoints. 
\end{parag}
	\begin{defn}\
		\label{def:truncationindeggeqp}
		\begin{enumerate}[label=\arabic{enumi}., ref=\thelemman.\arabic{enumi}]
			\item \label{def:naivetruncationindeggeqp}The \infinity-functor\begin{align*}
			\sigma_{\geqslant p}\colon \Modmixgr_{\Bbbk}\longrightarrow\varepsilon\operatorname{-}\Mod^{{\operatorname{gr},\geqslant p}}_{\Bbbk}
			\end{align*}\textit{left} adjoint to the inclusion $$\iota_{\geqslant p}\colon\varepsilon\operatorname{-}\Mod^{{\operatorname{gr},\geqslant p}}_{\Bbbk}\subseteq\Modmixgr_{\Bbbk}$$
			is the \textit{naive truncation in weights $\geqslant p$}.
			\item \label{def:clevertruncationindeggeqp}The \infinity-functor\begin{align*}
			\theta_{\geqslant p}\colon \Modmixgr_{\Bbbk}\longrightarrow\varepsilon\operatorname{-}\Mod^{{\operatorname{gr},\geqslant p}}_{\Bbbk}
			\end{align*}\textit{right} adjoint to the inclusion $$\iota_{\geqslant p}\colon\varepsilon\operatorname{-}\Mod^{{\operatorname{gr},\geqslant p}}_{\Bbbk}\subseteq\Modmixgr_{\Bbbk}$$
			is the \textit{naive truncation in weights $\geqslant p$}.
		\end{enumerate}
	\end{defn}
	Again, the notation and the nomenclature suggest the behavior of the \infinity-functors on a mixed graded $\Bbbk$-module: the one in Definition \ref{def:naivetruncationindeggeqp} simply kills $M_q$ for any $q<p$, while the one in Definition \ref{def:clevertruncationindeggeqp} changes the underlying graded $\Bbbk$-module even in weight $q\geqslant p$, since it involves some sort of totalization.
\subsection{The $t$-structure on mixed graded modules}
\label{sec:tstruct}
The main result of this section is the following Theorem, which provides a $t$-structure on mixed graded $\Bbbk$-modules describing explicitly both connective and coconnective objects.
\begin{theorem}
	\label{thm:mixedgradedtstructure}
	Let $\lp\Modmixgr_{\Bbbk}\rp_{\!\geqslant 0}$ be the full $\infinity$-subcategory of $\Modmixgr_{\Bbbk}$ spanned by those mixed graded $k$-modules $M_{\bullet}$ such that, for any integer $q$, the $\Bbbk$-module $M_q$ is $(-q)$-connective. Dually, let $\lp\Modmixgr_{\Bbbk}\rp_{\!\leqslant 0}$ be the full $\infinity$-subcategory spanned by those mixed graded ${\Bbbk}$-modules $M_{\bullet}$ such that, for all $q\in\ZZ$, the $\Bbbk$-module $M_q$ is $(-q)$-coconnective. These sub-$\infinity$-categories determine an accessible $t$-structure on the stable $\infinity$-category of mixed graded ${\Bbbk}$-modules, which we call the \textit{mixed graded Postnikov $t$-structure}.
	When ${\Bbbk}$ is discrete, the heart of such $t$-structure is equivalent to the classical abelian $1$-category of chain complexes of ${\Bbbk}$-modules $\dgMod_{\Bbbk}$.\end{theorem}
\begin{remark}
	The name for such $t$-structure is borrowed from \cite{tstructurektheory}, where the standard $t$-structure on modules over a connective $\mathbb{E}_1$-ring of \cite[Proposition $7.1.1.13$]{ha} is referred to as the \textit{Postnikov $t$-structure}.
\end{remark}
\begin{proof}[Proof of \cref{thm:mixedgradedtstructure}]
	Consider the sub-\infinity-category $\lp\Modmixgr_{\Bbbk}\rp_{\!\geqslant0}$ described in the statement of \cref{thm:mixedgradedtstructure}. It is a sub-\infinity-category of $\Modmixgr_{\Bbbk}$ which is closed under all colimits and extensions, since they are computed weight-wise (\cref{lemma:colimitsinmixgr}) and so it suffices to prove the claim in $\Mod_{\Bbbk}$, where $q$-connective objects are stable under colimits and extensions for any $q$ (this is \cite[Proposition $7.1.1.13$]{ha}). Moreover it is presentable: indeed, it fits in a $(\infinity,2)$-pullback diagram of presentable \infinity-categories$$\begin{tikzpicture}[scale=0.75]
	\node (a) at (-4,3){$\lp\Modmixgr_{\Bbbk}\rp_{\!\geqslant0}$};
	\node (b) at (1,3){$\Modmixgr_{\Bbbk}$};
	\node (c) at (-4,0){$\lp\Modgr_{\Bbbk}\rp_{\!\geqslant0}$};
	\node (d)at (1,0){$\Modgr_{\Bbbk}$};
	\node at (-3.4,2.2){$\lrcorner$};
	\draw[->, font=\scriptsize] (a) to node[left]{} (b);
	\draw[->, font=\scriptsize] (a) to node[left]{} (c);
	\draw[->, font=\scriptsize] (b) to node[right]{$\oblv_{\varepsilon}$} (d);
	\draw[right hook->, font=\scriptsize] (c) to node[right]{} (d);
	\end{tikzpicture}$$
	where $$\lp\Modgr_{\Bbbk}\rp_{\!\geqslant0}\coloneqq\prod_{q\in\ZZ}\lp\Mod_{\Bbbk}\rp_{\geqslant-q}.$$Thus, the sub-\infinity-category $\lp\Modmixgr_{\Bbbk}\rp_{\geqslant0}$ defines a $t$-structure such that $\lp\Modmixgr_{\Bbbk}\rp_{\geqslant0}$ is precisely the \infinity-category of connective objects, thanks to \cite[Proposition $1.4.4.11$]{ha}. In detail: under these hypotheses, the \infinity-category $\lp\Modmixgr_{\Bbbk}\rp_{\!\geqslant0}$ is a colocalization of $\Modmixgr_{\Bbbk}$, hence the Adjoint Functor Theorem provides an \infinity-functor
	\begin{align}
	\label{functor:connectivecover}
	\tau^{\mixgr}_{\geqslant 0}\colon\Modmixgr_{\Bbbk}\longrightarrow\lp\Modmixgr_{\Bbbk}\rp_{\geqslant0}
	\end{align}right adjoint to the obvious inclusion. Thus, the $(-1)$-coconnective objects are defined to be those mixed graded $\Bbbk$-modules $N_{\bullet}$ such that $\tau^{\mixgr}_{\geqslant 0}N_{\bullet}\simeq 0$ in $\lp\Modmixgr_{\Bbbk}\rp_{\!\geqslant 0}$, and so each mixed graded $\Bbbk$-module is sandwiched in a (\infinity-functorial) cofiber sequence
	\begin{align}
	\label{sequence:fiberseqmixedgraded}
	\tau^{\mixgr}_{\geqslant 0}M_{\bullet}\longrightarrow M_{\bullet}\longrightarrow \tau^{\mixgr}_{\leqslant -1}M_{\bullet}
	\end{align}
	where $\tau^{\mixgr}_{\geqslant 0}M_{\bullet}\to M_{\bullet}$ is the adjoint to the identity of $\tau^{\mixgr}_{\geqslant 0}M_{\bullet}$. The interesting part of the statement of the Theorem is the one that characterizes the coconnective objects: we have to prove that a mixed graded $\Bbbk$-module is coconnective $N_{\bullet}$ \textit{precisely} if $N_q$ is $(-q)$-coconnective for any integer $q$. \cite[Remark $1.2.1.3$]{ha} assures us that $(-1)$-coconnective objects are uniquely determined by the property of being right orthogonal to any connective object. Since the enrichment in spaces of a $\Bbbk$-linear \infinity-category $\scrC$ is given by the truncation in degrees $\geqslant 0$ of the mapping $\Bbbk$-module, this means that the $\Bbbk$-module $\Map_{\Modmixgr_{\Bbbk}}{\lp M_{\bullet},\hsp N_{\bullet}\rp}$ has to be $(-1)$-coconnective for any $M_{\bullet}$ connective. Since in $\Modmixgr_{\Bbbk}$ colimits are computed weight-wise, and by definition$$\lp\Modmixgr_{\Bbbk}\rp_{\leqslant-1}\coloneqq\lp\Modmixgr_{\Bbbk}\rp_{\leqslant0}[-1],$$our claim is equivalent to proving that $$\tau_{\geqslant 0}\Map_{\Modmixgr_{\Bbbk}}{\lp M_{\bullet},\hsp N_{\bullet}\rp}\simeq \{*\}$$ for $N_{\bullet}$ such that $N_q$ is $(-q-1) $-coconnective for any integer $q$. We shall need the following two Lemmas.
	%Let us remark that the mapping space above provides the enrichment in \textit{spaces}, and is obtained by the connective truncation of the mapping $\Bbbk$-module presented in \ref{parag:monoidalstructureonmodmixgr}. 
	\begin{lemman}
		\label{lemma:sufficiency}
		Let $N_{\bullet}$ such that $N_q$ is $(-q-1)$-coconnective for any integer $q$. Then, $N_{\bullet}$ is a $(-1)$-coconnective object for the mixed graded Postnikov $t$-structure. 
	\end{lemman}
	\begin{proof}
		We recall that the enrichment in $\Bbbk$-modules of mixed graded $\Bbbk$-modules is given by the fiber for the mixed structure morphism$$	{\Mapmixgr_{{\Bbbk}}}{\lp M_{\bullet},\hsp N_{\bullet}\rp}_0\overset{\varepsilon_0}{\longrightarrow}	{\Mapmixgr_{{\Bbbk}}}{\lp M_{\bullet},\hsp N_{\bullet}\rp}_{-1}[-1].$$The enrichment in spaces is now given by applying the truncation \infinity-functor $\tau_{\geqslant 0}$ to such $\Bbbk$-module (\cite[Construction $1.3.1.13$]{ha}). So, let $M_{\bullet}$ be such that $M_q$ is $(-q)$-connective for all $q$'s: we have a fiber sequence$$\Map_{\Modmixgr}{\lp M_{\bullet},\hsp N_{\bullet}\rp}\longrightarrow{\Mapmixgr_{{\Bbbk}}}{\lp M_{\bullet},\hsp N_{\bullet}\rp}_0\overset{\varepsilon_0}{\longrightarrow}	{\Mapmixgr_{{\Bbbk}}}{\lp M_{\bullet},\hsp N_{\bullet}\rp}_{-1}[-1].$$Writing explicitly the second and third terms of such fiber sequence, using \cref{parag:monoidalstructureonmodmixgr}, we get$$\Map_{\Modmixgr}{\lp M_{\bullet},\hsp N_{\bullet}\rp}\longrightarrow\prod_{q\in\ZZ}{\Map_{\Mod_{\Bbbk}}}{\lp M_q,\hsp N_q\rp}\overset{\varepsilon}{\longrightarrow} \prod_{p\in\ZZ}{\Map_{\Mod_{\Bbbk}}}{\lp M_p,\hsp N_{p-1}\rp}[-1].$$
		The second term of such fiber sequence is (at least) $(-1)$-coconnective, because it is the product of $(-1)$-coconnective $\Bbbk$-modules (being $M_q$ $(-q)$-connective and $N_q$ $(-q-1)$-coconnective for all $q$'s). Analogously, also the third term is (at least) $(-2)$-coconnective, hence in particular $(-1)$-coconnective. Being $(-1)$-coconnective objects stable under all limits, we get that $\Map_{\Modmixgr_{\Bbbk}}{\lp M_{\bullet},\hsp N_{\bullet}\rp}$ is now $(-1)$-coconnective, hence its truncation yields a contractible space.
	\end{proof}
	\begin{lemman}
		\label{lemma:necessity}
		Suppose that $N_{\bullet}$ is $(-1)$-coconnective for the mixed graded Postnikov $t$-structure. Then $N_q$ is $(-q-1)$-coconnective as a $\Bbbk$-module.
	\end{lemman}
	\begin{proof}
		Without loss of generality, by shifting weights, we can reduce ourselves to the case that $N_{\bullet}$ is such that, for some integer $n\geqslant0$, there exists a non-trivial homotopy group $\pi_n N_0\not\cong 0$. Let $\Bbbk(0)$ be the mixed graded $\Bbbk$-module with trivial mixed structure in pure weight $0$. Then by the adjunction \ref{adjunction:realizationadjunction}, we have equivalences of $\Bbbk$-modules$$\Map_{\Modmixgr_{\Bbbk}}{\lp \Bbbk(0),\hsp N_{\bullet}\rp}\simeq\Map_{\Mod_{\Bbbk}}{\lp \Bbbk,\hsp \left|N_{\bullet}\right|\rp}\simeq\left|N_{\bullet}\right|.$$If the realization of $N_{\bullet}$ is not $(-1)$-truncated, then $N_{\bullet}$ cannot be $(-1)$-coconnective for the mixed graded Postnikov $t$-structure, since $\Bbbk(0)$ is obviously connective. So, let us assume that $\left|N_{\bullet}\right|$ is $(-1)$-coconnective as a $\Bbbk$-module. Using the naive truncation \infinity-functor in degree $\leqslant-1$ of \ref{def:naivetruncationindegleqp}, by adjunction we have a cofiber sequence$$\sigma_{\leqslant-1}N_{\bullet}\longrightarrow N_{\bullet} \longrightarrow \cofib.$$Applying the realization $\infinity$-functor, that being a right adjoint between stable \infinity-categories preserves cofiber sequences, we have another cofiber sequence of the form$$\left|\sigma_{\leqslant-1}N_{\bullet}\right|\longrightarrow\left|N_{\bullet}\right|\longrightarrow\left|\cofib\right|.$$
		Since limits and colimits are computed weight-wise in $\Modmixgr_{\Bbbk}$, by the description of the realization \infinity-functor provided in \ref{parag:modeladjunction}, we easily have that $\left|\cofib\right|\simeq N_0$, and that moreover $\left|\sigma_{\leqslant-1}N_{\bullet}\right|\simeq N_0[-1].$ In particular, $\left|\sigma_{\leqslant-1}N_{\bullet}\right|$ cannot be trivial, because this would contradict that $\pi_nN_0\not\cong 0$ for some $n\geqslant0.$ Therefore, considering the $\Bbbk$-module $\Bbbk[1](-1)$, we have\begin{align*}
		\Map_{\Modmixgr_{\Bbbk}}{\lp \Bbbk[1](-1),\hsp N_{\bullet}\rp}&\simeq\Map_{\Modmixgr_{\Bbbk}}{\lp \Bbbk(0),\hsp \sigma_{\leqslant-1}N_{\bullet}[-1]\lpd 1\rpd\rp}\\
		&\simeq\Map_{\Mod_{\Bbbk}}{\lp \Bbbk,\hsp \left|\sigma_{\leqslant-1}N_{\bullet}[-1]\lpd 1\rpd\right|\rp}\\
		&\simeq \left|\sigma_{\leqslant-1}N_{\bullet}\lpd 1\rpd[-1]\right|\simeq N_0.
		\end{align*}
		Now, $\Bbbk[1](-1)$ is again connective for the mixed graded Postnikov $t$-structure, and now we have a non-trivial map towards $N_{\bullet}$, which is given by any non-trivial homology class in $\pi_nN_0.$	
	\end{proof}
	\cref{lemma:sufficiency} and \cref{lemma:necessity} together imply that the class of $(-1)$-coconnective objects for the mixed graded Postnikov $t$-structure and the class of mixed graded $\Bbbk$-modules which are $(-q-1)$-coconnective in each weight $q$ coincide, and this concludes the proof of the main statement. 	The claim about the heart of such $t$-structure can be proved directly. The idea is that a mixed graded $\Bbbk$-module lying in the heart of the $t$-structure is forced to be a mixed graded $\Bbbk$-module whose weight $p$ component is a discrete $\Bbbk$-module concentrated in homological degree $-p$, and the differential for this chain complex is given by the mixed differential. Yet, we shall see this claim as a corollary of \cref{thm:leftcompletion}.
\end{proof}
\begin{porism}
	The proof of \cref{thm:mixedgradedtstructure} \textit{does not rely} on the assumption that $\Bbbk$ is a field. We can always define the \infinity-category of mixed graded $R$-modules, for $R$ a $\QQ$-algebra in $\mathbb{E}_{\infty}$-rings, as the \infinity-category of quasi-coherent sheaves on the affine group stack $\mathsf{B}\mathbb{G}_{\operatorname{a},R}\rtimes\mathbb{G}_{\operatorname{m},R}$. Let $\Sp^{\gr}\coloneqq\prod_{p\in\ZZ}\Sp$ be the \infinity-category of graded spectra, and let $\Sp^{\gr}_{\geqslant0}$ be its full sub-\infinity-category spanned by graded spectra which are $(-q)$-connective for all integers $q$. In virtue of \cite[Theorem $4.1$]{geometryoffiltrations}, we have an equivalence of \infinity-categories$$\Sp^{\gr}\simeq\Qcoh{\lp\mathsf{B}\mathbb{G}_{\operatorname{m},\mathbb{S}}\rp},$$where $\mathbb{G}_{\operatorname{m},\mathbb{S}}\coloneqq\Spec(\mathbb{S}[\ZZ])$ is the flat multiplicative group (spectral) scheme. Hence, by the abstract nonsense recollected at the beginning of the proof of \cref{thm:mixedgradedtstructure}, the \infinity-category$$\lp\Modmixgr_{R}\rp_{\geqslant 0}\coloneqq \Modmixgr_{R}\times_{\Sp^{\gr}}\Sp^{\gr}_{\geqslant0}$$is again presentable, closed under all colimits and extensions, and so it determines the connective part for an accessible $t$-structure over $\Modmixgr_{R}$. Even for an arbitrary derived scheme $X$ defined over a derived $\QQ$-algebra $\Bbbk$, one can define (\cite[Section $1.2.1$]{derivedfoliations2}) the \textit{\infinity-category of mixed graded modules over $X$} as the \infinity-category of sheaves with values in $\Modmixgr_{\Bbbk}$ for the Zariski topology on $X$:
	$$\Modmixgr_X\coloneqq\operatorname{Shv}_{\operatorname{Zar}}{\lp X,\hsp \Modmixgr_{\Bbbk}\rp}.$$
	Setting $\lp\Modmixgr_X\rp_{\geqslant0}$ to be the $\infinity$-category of those sheaves of mixed graded $\Bbbk$-modules over $X$ which are connective as mixed graded $S$-modules for any $S$-point $\Spec(S)\to X$ (the fact that this yields, indeed, a $t$-structure is essentially \cite[Chapter $3$, Section $1.5.1$]{studyindag1}). The only property of a field of characteristic $0$ that we use in the proof of \cref{thm:mixedgradedtstructure} is the fact that it is connective: indeed, \textit{whenever $R$ is connective}, then the coconnective part for such $t$-structure can be characterized as$$\lp\Modmixgr_R\rp_{\leqslant 0}\coloneqq\Modmixgr_R\times_{\Sp^{\gr}}\Sp^{\gr}_{\leqslant0}.$$This fact is not true in general if $R$ is not connective (see also \cite[Warning $3.5.9$]{dagx}). The analogous statement on the coconnective part for the $t$-structure on mixed graded modules over a non-affine derived scheme $X$ case can be recovered under the hypothesis that the sheaf of $\Einf$-rings is connective (in the sense of \cite[Definition $1.20$]{dagvii}), using \cite[Proposition $2.1.3$]{dagviii}, or in the non-connective case when $X$ is at least an Artin stack (here, we use \cite[Chapter $3$, Proposition $1.5.4$]{studyindag1}). 
\end{porism}
The characterization of the connective and coconnective objects of the $t$-structure of \cref{thm:mixedgradedtstructure} allows us to describe pretty neatly how the connective cover and the coconnective cover \infinity-functors behave on a mixed graded $\Bbbk$-module.
\begin{corollaryn}
	\label{cor:truncationisgood}
	Let $\tau_{\geqslant 0}^{\mixgr}$ be the connective cover \infinity-functor described in \ref{functor:connectivecover}, and let $\tau_{\leqslant0}^{\mixgr}$ be the cofiber \infinity-functor described in \ref{sequence:fiberseqmixedgraded}. For any mixed graded $\Bbbk$-module $M_{\bullet}$, there exist natural equivalences of $\Bbbk$-modules $$\lp\tau^{\mixgr}_{\geqslant 0}M_{\bullet}\rp_q\simeq \tau_{\geqslant -q}M_q$$and$$\lp\tau_{\leqslant0}^{\mixgr}M_{\bullet}\rp_q\simeq \tau_{\leqslant-q}M_q.$$
\end{corollaryn}
\begin{proof}
	Since colimits are computed weight-wise, taking the $\Bbbk$-module in weight $q$ is an exact $\infinity$-functor from mixed graded $\Bbbk$-modules to $\Bbbk$-modules for any integer $q$. So, considering the canonical fiber sequence$$\tau_{\geqslant 0}^{\mixgr}M_{\bullet}\longrightarrow M_{\bullet} \longrightarrow\tau_{\leqslant-1}^{\mixgr} M_{\bullet}$$and taking the component in weight $q$ we get another fiber sequence$$\lp\tau_{\geqslant 0}^{\mixgr}M_{\bullet}\rp_q\longrightarrow M_q\longrightarrow\lp\tau_{\leqslant-1}^{\mixgr}M_{\bullet}\rp_q.$$By suitably shifting, it follows that this fiber sequence is the essentially unique fiber sequence of $\Bbbk$-modules that extends $M_q$ with a $(-q)$-connective part on the left and a $(-q-1)$-coconnective part on the right, hence $\lp\tau_{\geqslant 0}^{\mixgr}M_{\bullet}\rp_q\simeq \tau_{\geqslant -q}M_q$ and $\lp\tau_{\leqslant -1}^{\mixgr}M_{\bullet}\rp_q\simeq \tau_{\leqslant-q-1}M_q.$
\end{proof}
\section{Relationship with filtered modules}
\label{sec:filtered/mixedgraded}
In this section, we shall state the main link between the $\infinity$-categories of filtered $\Bbbk$-modules and mixed graded $\Bbbk$-modules, and their respective $t$-structures described in Theorems \ref{thm:beilinsontstructure} and \ref{thm:mixedgradedtstructure}. The main result of this section is \cref{thm:leftcompletion}, which states that complete (in the homotopical sense) filtrations on $\Bbbk$-modules arise uniquely from a mixed structure on the associated graded of the filtration, and moreover that this assignment is $t$-exact. This can be seen as a slight improvement on \cite[Proposition $1.3.1$]{derivedfoliations2} and \cite[Proposition $2.27$]{calaque2021lie}, which exhibited a fully faithful embedding $\Modmixgr_{\Bbbk}\hookrightarrow\Modfil_{\Bbbk}$ and characterized its essential image, without considering any $t$-structure on those $\infinity$-categories.
\subsection{Recollection on filtered modules}
We start by reviewing some important concepts concerning filtered $\Bbbk$-module, following \cite{BMS} and \cite{enhancingfiltered}.
\begin{defn}[{\cite[Definition $5.1$]{BMS}}]Let $\ZZ_{\geqslant}\coloneqq\ZZ^{\op}_{\leqslant}$ be the $1$-category on the poset $\ZZ$ with reverse order (i.e. there exists an arrow $n\to m$ whenever $n\geqslant m$). The \textit{$\infinity$-category of filtered $\Bbbk$-modules} is the $\infinity$-functor $\infinity$-category$$\Modfil_{\Bbbk}\coloneqq{\Fun}{\lp\ZZ_{\geqslant},\hsp\Mod_{\Bbbk}\rp.}$$
\end{defn}
\begin{notation}
	An object of $\Modfil_{\Bbbk}$ consists of a tower of morphisms
	$$\ldots\longrightarrow M_{p+1}\longrightarrow M_p\longrightarrow M_{p-1}\longrightarrow\dots$$
	where $M_p$ is a $\Bbbk$-module for any integer $p$.  We shall denote the object associated to such tower with $M_{\bullet}$, while the $\Bbbk$-module $M_p$ will be the $p$ weight component of $M_{\bullet}$.\\Consistently, a natural transformation between two filtered $\Bbbk$-modules $M_{\bullet}\to N_{\bullet}$  (that is, a sequence of morphisms $f_p\colon M_p\to N_p$ for any integer $p$ with obvious compatibility with the transition morphisms $M_p\to M_{p-1}$ and $N_p\to N_{p-1}$) shall be denoted by $f_{\bullet}$.
\end{notation}
\begin{parag}
	The selection of the $p$ weight component of a filtered $\Bbbk$-module $M_{\bullet}$ is functorial, i.e. there exists an $\infinity$-functor\begin{align}
	\label{functor:weightp}
	(-)_p\colon\Modfil_{\Bbbk}\longrightarrow\Mod_{\Bbbk},
	\end{align}
	informally described on objects by the assignation $M_{\bullet}\mapsto M_p$,  given by the precomposition of an $\infinity$-functor $M_{\bullet}\colon\ZZ_{\geqslant}\to \Mod_{\Bbbk}$ with the inclusion of the terminal $1$-category $\{ p\}\subseteq \ZZ_{\geqslant}$.
\end{parag}
\begin{construction}
	Given a filtered $\Bbbk$-module $M_{\bullet}$, we can do more than selecting its $p$ weight component; we can also consider the $(-)_{\scriptstyle\infty}$ and $(-)_{-\scriptstyle\infty}$ $\infinity$-functors, given respectively by\begin{align}
	\label{functor:limitfiltration}
	(-)_{\scriptstyle\infty}\colon\Modfil_{\Bbbk}\coloneqq {\Fun}{\lp\ZZ_{\geqslant},\hsp\Mod_{\Bbbk}	\rp}\overset{\lim}{\longrightarrow}\Mod_{\Bbbk}
	\end{align}and\begin{align}\label{functor:colimitfiltration}
	(-)_{-\scriptstyle\infty}\colon\Modfil_{\Bbbk}\coloneqq {\Fun}{\lp\ZZ_{\geqslant},\hsp\Mod_{\Bbbk}\rp}\xrightarrow{\colim}\Mod_{\Bbbk}.
	\end{align}\end{construction}
\begin{defn}
	\label{def:completefiltered}
	A filtered $\Bbbk$-module $M_{\bullet}$ is \textit{complete} if $M_{\scriptstyle\infty}\simeq 0$.
\end{defn}
The complete filtered $\Bbbk$-modules are gathered in a full sub-$\infinity$-category of $\Modfil_{\Bbbk}$, which we shall denote by $\Modfilcomp_{\Bbbk}$. 
\begin{parag}
	With the above notations, the $p$ weight part $M_p$ of a filtered $\Bbbk$-module $M_{\bullet}$ must be thought as the $p$-th part of an exhaustive filtration on the $\Bbbk$-module $M_{-\scriptstyle\infty}$, while the $\Bbbk$-module $M_{\scriptstyle\infty}$ determines whether such filtration is separated Hausdorff \textit{in the homotopical setting} (which is not always equivalent to classical one - see \cite[Remark $2.13$]{enhancingfiltered} for a classically separated filtration which is not complete in our sense). 
\end{parag}
\begin{defn}Given a filtered $\Bbbk$-module $M_{\bullet}$, consider the inclusion of $\Delta^1\simeq \{ p+1\to p\}$ in $\ZZ_{\geqslant}$. This provides an $\infinity$-functor $\Modfil_{\Bbbk}\coloneqq{\Fun}{\lp\ZZ_{\geqslant},\hsp \Mod_{\Bbbk}\rp}\to {\Fun}{\lp\Delta^1,\hsp\Mod_{\Bbbk}\rp}$ which selects the morphism $M_{p+1}\to M_p$ in a filtered $\Bbbk$-module $M_{\bullet}$. \\
	The \textit{$p$-th graded piece $\infinity$-functor} is the composition of the previous $\infinity$-functor with the cofiber $\infinity$-functor$$\Gr_p\colon\Modfil_{\Bbbk}\longrightarrow{\Fun}{\lp \Delta^1,\hsp\Mod_{\Bbbk}\rp}\xrightarrow{\cofib}\Mod_{\Bbbk}.$$
\end{defn}
\begin{parag}
	\label{parag:monoidalstructureonmodfil}The $\infinity$-category $\Modfil_{\Bbbk}$ is naturally a closed symmetric monoidal $\infinity$-category (\cite[§ $2.23$]{enhancingfiltered}); such  monoidal structure is inherited in some sense by the closed symmetric monoidal structure on $\Mod_{\Bbbk}$ via the Day convolution. We briefly recall some important properties and useful explicit constructions.
	\begin{enumerate}
		\item Given two filtered $\Bbbk$-modules $M_{\bullet}$ and $N_{\bullet}$, the tensor product $M_{\bullet}\otimes^{\fil}_{\Bbbk}N_{\bullet}$ is the filtered $\Bbbk$-module whose $p$-th weight component is given by the formula
		$$\lp M_{\bullet}\otimes_{\Bbbk}^{\fil}N_{\bullet}\rp_p\coloneqq\underset{{\substack{i+j\geqslant p\\i,j\in\ZZ_{\geqslant}}}}{\colim}\hsp M_i\otimes_{\Bbbk} N_j.$$ 
		\item The unit for $\otimes_{\Bbbk}^{\fil}$ is the sequence $\Bbbk^{\leqslant 0}$ given by $\Bbbk$ in non positive weights and by $0$ otherwise, with only identities and trivial morphisms, i.e. by the sequence$$\ldots=0= 0\to \Bbbk = \Bbbk = \Bbbk =\dots$$where the first copy of $\Bbbk$ sits in weight $0$.
		\item Given two filtered $\Bbbk$-modules $M_{\bullet}$ and $N_{\bullet}$, the internal mapping space ${\Map_{\Bbbk}^{\fil}}{\lp M_{\bullet},\hsp N_{\bullet}\rp}$ is the filtered $\Bbbk$-module whose $p$-th weight component is given by the end formula $$\lp{\Map_{\Bbbk}^{\fil}}{\lp M_{\bullet},\hsp N_{\bullet}\rp}\rp_p\coloneqq \int_{q\in\ZZ_{\geqslant}}{ \Map_{\Bbbk}}{\lp M_q,\hsp N_{p+q}\rp}.$$
	\end{enumerate}
\end{parag}
\begin{lemman}[{\cite[Lemma $5.2$]{BMS}}]\
	\label{lemma:scholze}
	\begin{enumerate}[label=(\arabic{enumi}), ref=\thelemman.\arabic{enumi}]
		\item The collection of $\infinity$-functors $\{\Gr_p\}_{p\in\ZZ}$ and $(-)_{\scriptstyle\infty}$ is jointly conservative on $\Modfil_{\Bbbk}$. On the sub-$\infinity$-category $\Modfilcomp_{\Bbbk}$, the $\infinity$-functors $\{\Gr_p\}_{p\in\ZZ}$ are already jointly conservative. 
		\item\label{lemma:scholze2}The inclusion $\Modfilcomp_{\Bbbk}\subseteq\Modfil_{\Bbbk}$ admits a left adjoint $$\widehat{(-)}\colon\Modfil_{\Bbbk}\longrightarrow\Modfilcomp_{\Bbbk},$$which sends $M_{\bullet}$ to its completion $\widehat{M}_{\bullet}$ described in its $p$ weight component by $$\widehat{M}_p\coloneqq{\cofib}{\lp M_{\scriptstyle\infty}\to M_p\rp}.$$Such completion $\infinity$-functor does not alter the graded pieces of $M_{\bullet}$.
		\item Both $\Modfil_{\Bbbk}$ and $\Modfilcomp_{\Bbbk}$ have all limits and colimits. On the former, both $\{(-)_p\}_{p\in\ZZ}$ and $\{\Gr_p\}_{p\in\ZZ}$ commute with all limits and colimits; on the latter, $\{\Gr_p\}_{p\in\ZZ}$ commute with all limits and colimits.
		\item \label{lemma:scholze4}There exists a (unique) closed symmetric monoidal structure on $\Modfilcomp_{\Bbbk}$ compatible with the one on $\Modfil_{\Bbbk}$ via the completion $\infinity$-functor.
		\item  \label{lemma:scholze5} For any $M_{\bullet}$ and $N_{\bullet}$ in $\Modfil_{\Bbbk}$ or $\Modfilcomp_{\Bbbk}$, we have an equivalence$$\Gr_p\lp M_{\bullet}\otimes_{\Bbbk}^{\fil}N_{\bullet}\rp\simeq\bigoplus_{i+j=p}\Gr_iM_{\bullet}\otimes_{\Bbbk}\Gr_jN_{\bullet}.$$
	\end{enumerate}
\end{lemman}
\cref{lemma:scholze} allows us to say something more about the sub-\infinity-category of complete filtered $\Bbbk$-modules: namely, it is stable.
\begin{propositionn}
	\label{prop:completearestable}
	The full sub-\infinity-category $\Modfilcomp_{\Bbbk}$ of $\Modfil_{\Bbbk}$ is stable.
\end{propositionn}
\begin{proof}
	The \infinity-category $\Modfilcomp_{\Bbbk}$ has a zero object, since the trivial filtration is clearly complete. Moreover, since their inclusion in all filtered $\Bbbk$-modules admits a left adjoint, it follows that all limits in $\Modfilcomp_{\Bbbk}$ are computed as in $\Modfil_{\Bbbk}$. So, we only need to prove that given a diagram of complete filtered $\Bbbk$-modules of the form$$M_{\bullet}\longleftarrow N_{\bullet}\longrightarrow P_{\bullet}$$its colimit $R_{\bullet}$ is again a complete filtered $\Bbbk$-module, i.e., $$\lim_{q\in\ZZ_{\geqslant}}R_q\simeq 0.$$Since limits and colimits in $\Modfil_{\Bbbk}$ are computed weight-wise, we have that$$\lim_{q\in\ZZ_{\geqslant}}R_q\simeq \lim_{q\in\ZZ_{\geqslant}}{\lp M_q\coprod_{N_q}P_q\rp}.$$Since in a stable \infinity-category finite colimits commute with all limits (\cite[Proposition $1.1.4.1$]{ha}), we can write$$\lim_{q\in\ZZ_{\geqslant}}R_q\simeq \lim_{q\in\ZZ_{\geqslant}}M_q\coprod_{\lim_{q\in\ZZ_{\geqslant}}N_q}\lim_{q\in\ZZ_{\geqslant}}P_q$$and now since $M_{\bullet}$, $N_{\bullet}$ and $P_{\bullet}$ are all complete, it follows that all the limits on the right hand side are $0$, hence the pushout $R_{\bullet}$ is complete as well. In particular, any diagram in $\Modfilcomp_{\Bbbk}$ is a pullback (resp. pushout) if and only if it is a pullback (resp. pushout) in $\Modfil_{\Bbbk}$, and being the latter stable it follows that $\Modfilcomp_{\Bbbk}$ has to be stable as well.
\end{proof}
We now state (without proof) this important and well known result on the $t$-structure on the (homotopy category of the) $\infinity$-category of filtered $\Bbbk$-modules. Note that this theorem \textit{does not need} the assumption that $\Bbbk$ is a classical ring in characteristic $0$.
\begin{theorem}[Beilinson $t$-structure, {\cite[Appendix A]{beilinsonderived}} and {\cite[Theorem $5.4$]{BMS}}]
	\label{thm:beilinsontstructure}
	%	\todo{Robalo: riferimento alla $t$-struttura sulle cose graduate?}
	Let $\lp\Modfil_{\Bbbk}\rp_{\!\geqslant 0}$ be the full $\infinity$-subcategory of $\Modfil_{\Bbbk}$ spanned by those filtered modules such that ${\Gr}_pM_{\bullet}$ is $(-p)$-connective for all integers $p$. Dually, let $\lp\Modfil_{\Bbbk}\rp_{\!\leqslant 0}$ be the full $\infinity$-subcategory spanned by those filtered modules such that $M_{p}$ is $(-p)$-coconnective for all integers $p$. Then these two $\infinity$-subcategories define a $t$-structure on $\Modfil_{\Bbbk}$ whose heart $\lp\Modfil_{\Bbbk}\rp^{\!\!\heartsuit}$ is equivalent, as an abelian $1$-category, to the usual $1$-category of chain complexes of ${\Bbbk}$-modules $\dgMod_{\Bbbk}$ when ${\Bbbk}$ is a discrete ring. 
\end{theorem}
\begin{remark}
	The $t$-structure of \cref{thm:beilinsontstructure} is induced by a $t$-structure on graded $\Bbbk$-modules in the following sense. The \infinity-category $\Modgr_{\Bbbk}$ is endowed with a $t$-structure described as follows: the connective objects are those graded $\Bbbk$-modules $M_{\bullet}$ such that $M_p$ is $(-p)$-connective (for the Postnikov $t$-structure on $\Mod_{\Bbbk}$) for all $p$'s, and dually the coconnective objects are those graded $\Bbbk$-modules $M_{\bullet}$ such that $M_p$ is $(-p)$-coconnective for all $p$'s. Such $t$-structure can be lifted via the \infinity-functor$$\Gr_{\bullet}\colon\Modfil_{\Bbbk}\longrightarrow\Modgr_{\Bbbk}$$to a $t$-structure on $\Modfil_{\Bbbk}$ that makes $\Gr_{\bullet}$ a $t$-exact \infinity-functor between stable \infinity-categories endowed with a $t$-structure. This is precisely the Beilinson $t$-structure.
\end{remark}
%\begin{remark}
%The $t$-structure of \cref{thm:beilinsontstructure} is 
%\end{remark}
\subsection{Left completeness of $t$-structures}
Mixed graded $\Bbbk$-modules are not only fully faithfully embedded in filtered $\Bbbk$-modules, but such inclusion is universal in some precise sense. Namely, the $\infinity$-category of mixed graded $\Bbbk$-modules with the $t$-structure of \cref{thm:mixedgradedtstructure} is the \textit{left completion} of the $\infinity$-category of filtered $\Bbbk$-modules with respect to the Beilinson $t$-structure of \cref{thm:beilinsontstructure}. In order to precisely formalize and prove this assertion, we need to recall some important concepts concerning $t$-structures and left completions.
\begin{parag}
	Let $\scrC$ be a stable $\infinity$-category endowed with a $t$-structure $\lp \scrC_{\geqslant0},\hsp \scrC_{\leqslant 0}\rp$. For all integers $n$, we have the $n$-connective and $n$-coconnective cover $\infinity$-functors $\tau_{\geqslant n}\colon \scrC\to\scrC_{\geqslant n}$ and $\tau_{\leqslant n}\colon\scrC\to\scrC_{\leqslant n}$. In particular, we have a tower
	\begin{align}
	\label{diagram:leftcompletion}
	\ldots\xrightarrow{\tau_{\leqslant n+1}}\scrC_{\leqslant n+1}\xrightarrow{\tau_{\leqslant n}}\scrC_{\leqslant n}\xrightarrow{\tau_{\leqslant n-1}}\scrC_{\leqslant n-1}\xrightarrow{\tau_{\leqslant n-2}}\ldots
	\end{align}
\end{parag}
\begin{defn}[{\cite[Section $1.2.1$]{ha}}]
	The \textit{left completion} $\widehat{\scrC}$ of $\scrC$ is the limit of the diagram \ref{diagram:leftcompletion}.\\
	We will say that $\scrC$ is \textit{left complete} if the canonical morphism $\scrC\to \widehat{\scrC}$ is an equivalence of stable $\infinity$-categories.
\end{defn}
Let us recall this important result concerning left completions.
\begin{propositionn}[{\cite[Proposition $1.2.1.17$]{ha}}]
	\label{prop:leftcompletions}
	Let $\scrC$ be a stable $\infinity$-category endowed with a $t$-structure. Then:
	\begin{enumerate}
		\item The left completion $\widehat{\scrC}$ is also stable.
		\item The left completion $\widehat{\scrC}$ is naturally endowed with a $t$-structure which can be described as follows. Given an identification of $\widehat{\scrC}$ with the full $\infinity$-subcategory of ${\Fun}{\lp \ZZ_{\leqslant}^{\op},\hsp \scrC\rp}$ spanned by those functors in which $F_n$ factors through $\scrC_{\leqslant n}$ and such that $F_m\to  F_n$ induces an equivalence $\tau_{\leqslant n}F_m\overset{\simeq}{\to} F_n$ for all $n\leqslant m$ in $\ZZ$, then the connective (resp. coconnective) objects of $\widehat{\scrC}$ are those $\infinity$-functors which factor through $\scrC_{\geqslant 0}$ (resp. $\scrC_{\leqslant 0}$).
		\item The canonical $\infinity$-functor $\scrC\to\widehat{\scrC}$ is exact and induces an equivalence $\scrC_{\leqslant 0}\overset{\simeq}{\to}\widehat{\scrC}_{\leqslant 0}$.
	\end{enumerate}
\end{propositionn}
\begin{propositionn}
	The stable \infinity-category $\Modmixgr_{\Bbbk}$ is left complete with respect to the mixed graded Postnikov $t$-structure of Theorem \ref{thm:mixedgradedtstructure}. 
\end{propositionn}
\begin{proof}
	In order to prove our claim, we use the following criterion.
	\begin{propositionn}[{\cite[Proposition $1.2.1.19$]{ha}}]
		\label{prop:leftcompletecriterion}
		Let $\scrC$ a stable $\infinity$-category equipped with a $t$-structure. Suppose that $\scrC$ has all countable products, and $\scrC_{\geqslant 0}$ is closed under countable products. Then $\scrC$ is left exact with respect to its $t$-structure if and only if $$\scrC_{\geqslant\scriptstyle\infty}\coloneqq\bigcap_{n\in\ZZ}\scrC_{\geqslant n}$$consists only of zero objects of $\scrC$.
	\end{propositionn}
	Obviously $\Modmixgr_{\Bbbk}$, being both complete and cocomplete, admits all countable products. Since all limits and colimits are computed weight-wise, a product of countably many mixed graded $\Bbbk$-modules $\left\{M_{\bullet}^{\alpha}\right\}_{\alpha\in\NN}$, all of which are connective for the mixed graded Postnikov $t$-structure, is again connective. Indeed, in each weight we have an equivalence of $\Bbbk$-modules $$\lp\prod_{\alpha\in\NN}M^{\alpha}_{\bullet}\rp_{\!\!q}\simeq\prod_{\alpha\in\NN}M^{\alpha}_q$$where all the $M^{\alpha}_q$'s are now $(-q)$-connective, hence the statement follows from the fact that $\Mod_{\Bbbk}$ is left complete by \cite[Proposition $7.1.1.13$]{ha}.
\end{proof}
\begin{parag}
	On the converse, recall that the $\infinity$-category $\Modfil_{\Bbbk}$ endowed with the Beilinson $t$-structure of Theorem \ref{thm:beilinsontstructure} is \textit{not left complete} (\cite[Section $5$]{BMS}). In fact, the objects lying in $\lp\Modfil_{\Bbbk}\rp_{\geqslant\scriptstyle\infty}$ are those filtered ${\Bbbk}$-modules $M_{\bullet}$ such that ${\Gr_n}{ M_{\bullet}}$ vanishes for all integers $n$, i.e. filtered ${\Bbbk}$-modules corresponding to essentially constant diagrams $ \ZZ_{\geqslant}\longrightarrow \Mod_{\Bbbk}$. In particular, for any non trivial $\Bbbk$-module $M$ the constant diagram on it is \infinity-connective without being $0$.
\end{parag}
Yet, we can characterize what the left completion of the Beilinson $t$-structure on filtered $\Bbbk$-modules looks like: this is \textit{precisely} $\Modfilcomp_{\Bbbk}$ (which justifies, \textit{a posteriori}, the choice of notation for such sub-\infinity-category). While this seems to be known to experts, to our knowledge there is no proof of this fact available in the existing literature.
\begin{propositionn}
	\label{prop:leftcompletionBeilinson}
	The full sub-\infinity-category $\Modfilcomp_{\Bbbk}$ of $\Modfil_{\Bbbk}$ is endowed with a $t$-structure which is provided by the restriction of the Beilinson $t$-structure to complete filtered $\Bbbk$-modules. Moreover, the completion \infinity-functor$$\widehat{(-)}\colon\Modfil_{\Bbbk}\longrightarrow\Modfilcomp_{\Bbbk}$$is a $t$-exact \infinity-functor that naturally identifies $\Modfilcomp_{\Bbbk}$ with the left completion of the Beilinson $t$-structure on all filtered $\Bbbk$-modules.
\end{propositionn}
For the sake of clarity, we shall split the proof of \cref{prop:leftcompletionBeilinson} in several lemmas.
	\begin{lemman}
	\label{lemma:completeBeilinsonrestricts}
	The Beilinson $t$-structure of \cref{thm:beilinsontstructure} on filtered $\Bbbk$-modules restricts to a $t$-structure on the full sub-\infinity-category of complete filtered $\Bbbk$-modules.
	\end{lemman}
In order to prove \cref{lemma:completeBeilinsonrestricts}, we shall need the following important technical result.
\begin{lemman}
	\label{lemma:coconnectiveimpliescomplete}
	If $M_{\bullet}$ is a filtered $\Bbbk$-module which is eventually coconnective (that is, $n$-coconnective for some $n$) for the Beilinson $t$-structure of \cref{thm:beilinsontstructure}, then $M_{\bullet}$ is complete.
\end{lemman}
\begin{proof}
	Since for any integer $p$, the $\Bbbk$-module $M_p$ is $(n-p)$-coconnective for the usual $t$-structure on $\Bbbk$-modules by assumption, and coconnective objects are stable under all limits, we just need to observe that $$M_{\infty}\coloneqq\lim_{q\in\ZZ_{\geqslant}}{M_q}$$ must be $n-q$-coconnective for each integer $q$, i.e., it must belong to the \infinity-category$$\lp\Mod_{\Bbbk}\rp_{\leqslant-\infty}\coloneqq\bigcap_{q\in\ZZ}\lp\Mod_{\Bbbk}\rp_{\leqslant -q}.$$But the Postnikov $t$-structure on $\Bbbk$-modules is left complete, hence $M_{\infty}$ must be trivial because of \cref{prop:leftcompletecriterion}.
\end{proof}
\begin{proof}[Proof of \cref{lemma:completeBeilinsonrestricts}]
Since $\Modfilcomp_{\Bbbk}$ is a stable full sub-\infinity-category of $\Modfil_{\Bbbk}$ and its inclusion admits a right adjoint, it is closed under all limits and finite colimits, and in particular under all loops and suspensions. Moreover, if $M_{\bullet}$ is a complete filtered $\Bbbk$-module which is connective for the Beilinson $t$-structure, it is obviously left orthogonal to any complete filtered $\Bbbk$-module which is $(-1)$-connective for the Beilinson $t$-structure, since $\Modfilcomp_{\Bbbk}$ is a \textit{full} sub-\infinity-category of $\Modfil_{\Bbbk}$.  So, we only need to check is that if $M_{\bullet}$ is endowed with a complete filtration, then both $\taufil_{\leqslant n}M_{\bullet}$ and $\taufil_{\geqslant n}M_{\bullet}$ lie in $\Modfilcomp_{\Bbbk}$, where $\taufil_{\geqslant n}$ and $\taufil_{\leqslant n}$ denote the $n$-connective and $n$-coconnective cover \infinity-functor for the Beilinson $t$-structure on $\Modfil_{\Bbbk}$, respectively. This is clear for $\taufil_{\leqslant n}$, since $n$-coconnective objects are always complete in virtue of \cref{lemma:coconnectiveimpliescomplete}, and this implies that the same holds also for $\tau_{\geqslant n}$: indeed, consider the left adjoint to the inclusion $\Modfilcomp_{\Bbbk}\subseteq \Modfil_{\Bbbk}$ of \cref{lemma:scholze2}. For any filtered module $M_{\bullet}$, the unit map of the adjunction yields a map of fiber sequences (given by the unit map of the adjunction)$$\begin{tikzpicture}[scale=0.75]
\node (a1) at (-4,2){$\taufil_{\geqslant n+1}M_{\bullet}$};
\node (a2) at (0,2){$M_{\bullet}$};
\node (a3) at (4,2){$\taufil_{\leqslant n}M_{\bullet}$};
\node (b1) at (-4,0){$\widehat{\taufil_{\geqslant n+1}M_{\bullet}}$};
\node (b2) at (0,0){$\widehat{M_{\bullet}}$};
\node (b3) at (4,0){$\widehat{\taufil_{\leqslant n}M_{\bullet}}.$};
\draw[->,font=\scriptsize](a1) to node[above]{} (b1);
\draw[->,font=\scriptsize](a2) to node[above]{} (b2);
\draw[->,font=\scriptsize](a3) to node[above]{} (b3);
\draw[->,font=\scriptsize](a1) to node[above]{} (a2);
\draw[->,font=\scriptsize](a2) to node[above]{} (a3);
\draw[->,font=\scriptsize](b1) to node[above]{} (b2);
\draw[->,font=\scriptsize](b2) to node[above]{} (b3);
\end{tikzpicture}$$If $M_{\bullet}$ is complete, then both the second and third vertical arrows are equivalences, and so the first must be an equivalence as well. It follows that both the connective and coconnective cover \infinity-functors restrict naturally to the stable full sub-\infinity-category $\Modfilcomp_{\Bbbk}$, which implies our assertion.
\end{proof}
\begin{lemman}
	\label{lemma:completeBeilinsonleftcomplete}
The restriction of the Beilinson $t$-structure to the \infinity-category of complete filtered $\Bbbk$-modules is a left complete $t$-structure.
\end{lemman}
\begin{proof}
Since $\Modfilcomp_{\Bbbk}$ is closed under all limits existing in $\Modfil_{\Bbbk}$, the hypotheses of the left completeness criterion of \cref{prop:leftcompletecriterion} are satisfied, so we can simply check what the \infinity-connective objects of $\Modfilcomp_{\Bbbk}$ are: they are \infinity-connective objects of $\Modfil_{\Bbbk}$ (hence, constant filtrations) which are also complete as filtered $\Bbbk$-modules. But it is immediate to see that the only constant filtration which is complete is the constant filtration on the trivial $\Bbbk$-module $0$. Thus, $\Modfilcomp_{\Bbbk}$ is left complete.
\end{proof}
\begin{lemman}
\label{lemma:completeBeilinsonleftcompletion}
The left completion $\infinity$-functor $$\widehat{(-)}\colon\Modfil_{\Bbbk}\longrightarrow\Modfilcomp_{\Bbbk},$$identifies $\Modfilcomp_{\Bbbk}$, endowed with the restriction of the Beilinson $t$-structure, with the left completion of the Beilinson $t$-structure on $\Modfil_{\Bbbk}$.
\end{lemman}
\begin{proof}
In virtue of \cref{lemma:completeBeilinsonleftcomplete}, the restriction of the Beilinson $t$-structure on complete filtered $\Bbbk$-modules is a left complete $t$-structure. This means, in virtue of \cite[Proposition $1.2.1.17$]{ha}, that the natural \infinity-functor from the \infinity-category $\Modfilcomp_{\Bbbk}$ towards its left completion is an equivalence. As already described in \cref{prop:leftcompletions}, we can identify the left completion of the Beilinson $t$-structure on $\Modfilcomp_{\Bbbk}$ with the full sub-\infinity-category $\widehat{\Fun}{\lp\ZZ_{\leqslant},\hsp\Modfilcomp_{\Bbbk}\rp}\subseteq\Fun{\lp\ZZ_{\leqslant},\hsp\Modfilcomp_{\Bbbk}\rp}$ spanned by those \infinity-functors $F\colon\ZZ_{\leqslant}\to\Modfilcomp_{\Bbbk}$ such that\begin{itemize}
	\item $F_n\in\lp\Modfilcomp_{\Bbbk}\rp_{\leqslant -n}$ for any integer $n$.
	\item For any $m\leqslant n$, the map $F_m\to F_n$ induces an equivalence $\taufil_{\leqslant-n}F_m\overset{\simeq}{\to} F_n$.
\end{itemize} 
Employing this canonical model, the canonical equivalence $\Modfilcomp_{\Bbbk}\to\widehat{\Fun}{\lp\ZZ_{\leqslant},\hsp\Modfilcomp_{\Bbbk}\rp}$ simply sends an object $M_{\bullet}$ to the tower of its truncations $\left\{ \taufil_{\leqslant-n}M_{\bullet}\right\}_{n\in\ZZ}$, while its inverse takes the limit of a functor over $\ZZ_{\leqslant}$. So, we have a diagram
$$\begin{tikzpicture}[scale=0.75]
\node (a) at (-2,2){$\Modfilcomp_{\Bbbk}$};
\node (b) at (-2,0){$\Modfil_{\Bbbk}$};
\node (c) at (2,2){$\widehat{\Fun}{\lp\ZZ_{\leqslant},\hsp\Modfilcomp_{\Bbbk}\rp}$};
\node (d) at (2,0){$\widehat{\Fun}{\lp\ZZ_{\leqslant},\hsp\Modfil_{\Bbbk}\rp}$};
\draw[->,font=\scriptsize] (a) to node[above]{$\simeq$} (c);
\draw[right hook->, font=\scriptsize] (a) to node[left]{}(b);
\draw[->,font=\scriptsize] (c) to node[right]{}(d);
\draw[->,font=\scriptsize] (b) to node[below]{} (d);
\end{tikzpicture}$$where the vertical arrow on the right is simply the post-composition with the natural inclusion. Since the coconnective truncation \infinity-functors for filtered modules and complete filtered modules are the same, this diagram commutes; moreover, the functor $$\widehat{\Fun}{\lp\ZZ_{\leqslant},\hsp\Modfilcomp_{\Bbbk}\rp}\longrightarrow\widehat{\Fun}{\lp\ZZ_{\leqslant},\hsp\Modfil_{\Bbbk}\rp}$$is an equivalence as well, because a functor $F\colon\ZZ_{\leqslant}\to\Modfil_{\Bbbk}$ lies in $\widehat{\Fun}{\lp\ZZ,\hsp\Modfil_{\Bbbk}\rp}$ if and only if $F_n$ is $(-n)$-coconnective for all $n$. In particular, using again \cref{lemma:coconnectiveimpliescomplete}, it follows that at least as mere stable \infinity-categories, $\Modfilcomp_{\Bbbk}\simeq\widehat{\Fun}{\lp\ZZ_{\leqslant},\hsp\Modfil_{\Bbbk}\rp}.$\\The last thing we need to prove is that this equivalence is $t$-exact. Recall (\cite[Proposition $1.2.1.17.(2)$]{ha}) that the natural $t$-structure on $\widehat{\Fun}{\lp\ZZ_{\leqslant},\hsp\Modfil_{\Bbbk}\rp}$ is defined as follows: the connective objects are those \infinity-functors $F\colon\ZZ_{\leqslant}\to\Modfil_{\Bbbk}$ that factor through $\lp\Modfil_{\Bbbk}\rp_{\geqslant0}$, while coconnective objects are those \infinity-functors that factor through $\lp\Modfil_{\Bbbk}\rp_{\leqslant0}$; an analogous statement holds for $\widehat{\Fun}{\lp\ZZ_{\leqslant},\hsp\Modfilcomp_{\Bbbk}\rp}$. But, being complete with respect to its $t$-structure, the canonical \infinity-functor $$\Modfilcomp_{\Bbbk}\longrightarrow \widehat{\Fun}{\lp\ZZ_{\leqslant},\hsp\Modfilcomp_{\Bbbk}\rp}$$is a $t$-exact equivalence, and$$\widehat{\Fun}{\lp\ZZ_{\leqslant},\hsp\Modfilcomp_{\Bbbk}\rp}\longrightarrow\widehat{\Fun}{\lp\ZZ_{\leqslant},\hsp\Modfil_{\Bbbk}\rp}$$is again $t$-exact, because by definition$$\lp\Modfilcomp_{\Bbbk}\rp_{\geqslant0}\coloneqq \Modfilcomp_{\Bbbk}\times_{\Modfil_{\Bbbk}}\lp\Modfil_{\Bbbk}\rp_{\geqslant0}$$and the dual claim holds also for $\lp\Modfilcomp_{\Bbbk}\rp_{\leqslant0}.$ In particular, $\widehat{\Fun}{\lp\ZZ_{\leqslant},\hsp\Modfilcomp_{\Bbbk}\rp}\to\widehat{\Fun}{\lp\ZZ_{\leqslant},\hsp\Modfil_{\Bbbk}\rp}$ sends connective objects to connective objects, and coconnective objects to coconnective objects.\\
Now, since every complete filtered $\Bbbk$-module is canonically equivalent to the limit over the tower of its coconnective truncations, it is clear that the composition$$\Modfil_{\Bbbk}\longrightarrow\widehat{\Fun}{\lp\ZZ_{\leqslant},\hsp\Modfil_{\Bbbk}\rp}\simeq\widehat{\Fun}{\lp\ZZ_{\leqslant},\hsp\Modfilcomp_{\Bbbk}\rp}\simeq \Modfilcomp_{\Bbbk}$$agrees with the left adjoint to$$\Modfilcomp_{\Bbbk}\simeq\widehat{\Fun}{\lp\ZZ_{\leqslant},\hsp\Modfilcomp_{\Bbbk}\rp}\simeq\widehat{\Fun}{\lp\ZZ_{\leqslant},\hsp\Modfil_{\Bbbk}\rp}\longrightarrow\Modfil_{\Bbbk}$$which is canonically equivalent to the inclusion $\Modfilcomp_{\Bbbk}\subseteq\Modfil_{\Bbbk}$. Hence $\widehat{(-)}\colon\Modfil_{\Bbbk}\to\Modfilcomp_{\Bbbk}$ does exhibit the \infinity-category of complete filtered $\Bbbk$-modules as the left completion of the Beilinson $t$-structure on all filtered $\Bbbk$-modules.
\end{proof}
\subsection{Main theorem}
\label{subsec:mainthm}
We can now state the main result of this section, and of this paper as well.
\begin{theorem}
	\label{thm:leftcompletion}
	Let $\Modfilcomp_{\Bbbk}$ denote the full sub-\infinity-category of $\Modfil_{\Bbbk}$ spanned by $\Bbbk$-modules with complete filtration. There exists a fully faithful embedding$$(-)^{\fil}\colon\Modmixgr_{\Bbbk}\longhookrightarrow\Modfil_{\Bbbk}$$which is $t$-exact with respect to both $t$-structures, and which restricts to a $t$-exact equivalence$$\Modmixgr_{\Bbbk}\overset{\simeq}{\longrightarrow}\Modfilcomp_{\Bbbk}.$$
 Moreover, such embedding admits a $t$-exact left adjoint $$(-)^{\mixgr}\colon\Modfil_{\Bbbk}\longrightarrow\Modmixgr_{\Bbbk}$$ which makes the following diagram of \infinity-functors commute.
	\begin{displaymath}
	\begin{tikzpicture}[scale=0.75]
	\node(a) at (-2,2.5){$\Modfil_{\Bbbk}$};
	\node(b) at (2,2.5){$\Modfilcomp_{\Bbbk}$};
	\node (c) at (0,0.5){$\Modmixgr_{\Bbbk}$};
	\draw[->,font=\scriptsize] (a) to node[above]{$\widehat{(-)}$} (b);
	%	\draw[->,font=\scriptsize] (a) to node[below]{$\varphi$} (b);
	\draw[->,font=\scriptsize] (a) to[bend right] node[left]{$\lp-\rp^{\mixgr}$} (c);
	\draw[->,font=\scriptsize] (b) to[bend left] node[right]{$(-)^{\mixgr}$} (c);
	\end{tikzpicture}
	\end{displaymath}
In particular, the \infinity-functor $(-)^{\mixgr}\colon\Modfil_{\Bbbk}\to\Modmixgr_{\Bbbk}$ exhibits the \infinity-category $\Modmixgr_{\Bbbk}$, endowed with the mixed graded Postnikov $t$-structure of \cref{thm:mixedgradedtstructure}, as the left completion of $\Modfil_{\Bbbk}$ with respect to the Beilinson $t$-structure of \cref{thm:beilinsontstructure}.
	%	
	%	agrees with the 
	%	Let $\Modfilcomp_{\Bbbk}$ denote the left completion of $\Modfil_{\Bbbk}$ with respect to the Beilinson $t$-structure of \cref{thm:beilinsontstructure}. 
	%	
	%	Then there exists an equivalence of stable $\infinity$-categories $$\varphi\colon\Modmixgr_{\Bbbk}\overset{\simeq}{\longrightarrow}\Modfilcomp_{\Bbbk} $$
	%	Moreover, $\varphi$ is a $t$-exact $\infinity$-functor between stable $\infinity$-categories endowed with $t$-structures, where $\Modfilcomp_{\Bbbk}$ is endowed with its natural $t$-structure (see \cite[Proposition $1.2.1.17$]{ha}) and $\Modmixgr_{\Bbbk}$ is endowed with the mixed graded Postnikov $t$-structure Theorem \ref{thm:mixedgradedtstructure}.
\end{theorem}
Even if the $\infinity$-functors that exhibit $\Modmixgr_{\Bbbk}$ as the left completion of $\Modfil_{\Bbbk}$ are precisely the same as the ones considered in \cite{derivedfoliations2}, for the sake of clarity we recall all the details of their construction in this subsection.
\begin{parag}
	\label{parag:filteredconstruction}
	Consider the naive truncation $\infinity$-functor of mixed graded $\Bbbk$-modules of Definition \ref{def:naivetruncationindegleqp}. Thanks to \cref{remark:inclusionstruncationsmixedgraded}, for any mixed graded $\Bbbk$-module $M_{\bullet}$ and for any triple of integers $r\leqslant q \leqslant p$ we have natural morphisms in $\Modmixgr_{\Bbbk}$ $$\sigma_{\leqslant r}M_{\bullet}\longhookrightarrow\sigma_{\leqslant q}M_{\bullet}\longhookrightarrow\sigma_{\leqslant p}M_{\bullet}.$$
	This means that gathering all the naive truncations $\infinity$-functors $\{\sigma_{\leqslant p}\}_{p\in\ZZ}$ we get an $\infinity$-functor$$\Modmixgr_{\Bbbk}\xrightarrow{\left\{\sigma_{\leqslant p}\right\}_p}{\Fun}{\lp\ZZ_{\leqslant},\hsp\Modmixgr_{\Bbbk}\rp}.$$Using the isomorphism $\ZZ_{\leqslant}\cong \ZZ_{\geqslant}$ given by changing signs, we can write the $\infinity$-functor above as$$\sigma\coloneqq\left\{\sigma_{\leqslant -p}\right\}_{p\in\ZZ}\colon\Modmixgr_{\Bbbk}\longrightarrow{\Fun}{\lp\ZZ_{\geqslant},\hsp\Modmixgr_{\Bbbk}\rp}.$$
	Post-composing $\sigma$ with the Tate realization $\infinity$-functor of \cref{def:taterealization}, we land in the \infinity-category ${\Fun}{\lp\ZZ_{\geqslant},\hsp \Mod_{\Bbbk}\rp}\eqqcolon\Modfil_{\Bbbk}$.
\end{parag}
\begin{defn}
	\label{def:filteredassociated}
	The $\infinity$-functor$$(-)^{\fil}\colon\Modmixgr_{\Bbbk}\overset{\sigma}{\longrightarrow}{\Fun}{\lp\ZZ_{\geqslant},\hsp\Modmixgr_{\Bbbk}\rp}\xrightarrow{|-|^{\operatorname{t}}\circ-}\Modfil_{\Bbbk}$$is the \textit{associated filtered $\infinity$-functor}.
\end{defn}
\begin{remark}
	\label{remark:keyremarkfiltration}
	Given a mixed graded $\Bbbk$-module $M_{\bullet}$, its associated filtered $\Bbbk$-module is given in weight $p$ by the Tate realization of its truncation in weights $\leqslant -p$, i.e. $\left|\sigma_{\leqslant -p}M_{\bullet}\right|^{\operatorname{t}}$ and all the morphisms are the inclusions exhibited in \ref{parag:filteredconstruction}. Let us study in greater detail the $\Bbbk$-module $\left|\sigma_{\leqslant -p}M_{\bullet}\right|^{\operatorname{t}}$: for any $p$, this mixed graded $\Bbbk$-module $\sigma_{\leqslant -p}M_{\bullet}$ is bounded above (i.e., its weight components are $0$ for all $q>-p$), so its Tate realization is equivalent to
	\begin{displaymath}
	\underset{i\leqslant 0}{\colim}\hsp{\Map_{\Modmixgr_{\Bbbk}}}{\lp \Bbbk(-i)[2i],\hsp \sigma_{\leqslant -p}M_{\bullet}\rp}\simeq \begin{cases}
	{\Map_{\Modmixgr_{\Bbbk}}}{\lp \Bbbk(-p)[2p],\hsp \sigma_{\leqslant -p}M_{\bullet}\rp} &\text{if }p < 0\\
	{\Map_{\Modmixgr_{\Bbbk}}}{\lp \Bbbk(0),\hsp \sigma_{\leqslant -p}M_{\bullet}\rp}&\text{if }p\geqslant 0.
	\end{cases}
	\end{displaymath} 
	Either way, in terms of explicit models of chain complexes, this means that $$\lp M^{\fil}_{\bullet}\rp_p\simeq\prod_{q\geqslant p}M_{-q}[-2q]$$where the right hand side is endowed with the total differential. The transition maps, in terms of the above equivalence, are simply given by the obvious inclusions$$\dots\longhookrightarrow\prod_{q\geqslant p+1}M_{-q}[-2q]\longhookrightarrow\prod_{q\geqslant p}M_{-q}[-2q]\longhookrightarrow\prod_{q\geqslant p-1}M_{-q}[-2q]\longhookrightarrow\dots$$
\end{remark}
\begin{notation}
	In the following, we shall often denote the object $\lp M_{\bullet}\rp^{\fil}$ with the suggestive notation $\prod_{q \geqslant p}M_{-q}[-2q]$, leaving implicit the fact that this is not actually a \textit{product} of $\Bbbk$-modules.
\end{notation}
We prove some important properties of the associated filtered $\infinity$-functor.
\begin{propositionn}
	The associated filtered $\infinity$-functor of \cref{def:filteredassociated} is a $t$-exact $\infinity$-functor between stable $\infinity$-categories with $t$-structure. 
\end{propositionn}
\begin{proof}
	\cref{remark:keyremarkfiltration} allows us to describe neatly the graded pieces of the filtered $\Bbbk$-module $M_{\bullet}^{\fil}$ for any mixed graded $\Bbbk$-module $M_{\bullet}$. In fact, the cofiber sequence$$\lp M_{\bullet}^{\fil}\rp_{p+1}\longrightarrow\lp M_{\bullet}^{\fil}\rp_{p}\longrightarrow\Gr_pM_{\bullet}^{\fil}$$can be now represented by explicit models as
	$$\prod_{q\geqslant p+1}M_{-q}[-2q]\longhookrightarrow\prod_{q\geqslant p}M_{-q}[-2q]\longrightarrow\Gr_pM_{\bullet}^{\fil},$$where again the product has to be thought as endowed with the total differential, and by direct inspection of the long exact sequence of homotopy groups we get that $\Gr_pM_{\bullet}^{\fil}\simeq M_{-p}[-2p]$.\\So, let us assume that $M_{\bullet}$ was an $n$-connective mixed graded $\Bbbk$-module for the $t$-structure of \cref{thm:mixedgradedtstructure}: then, for any integer $p$, $M_p$ was an $(n-p)$-connective $\Bbbk$-module. This implies that $\Gr_p M_{\bullet}^{\fil}\simeq M_{-p}[-2p]$ is an $(n-p)$-connective $\Bbbk$-module as well, so $(-)^{\fil}$ preserves connective objects.\\
	On the converse, let $M_{\bullet}$ be an $n$-coconnective mixed graded $\Bbbk$-module. Then $\lp M_{\bullet}^{\fil}\rp_p$ is equivalent to a product of $\Bbbk$-modules sitting in $\lp\Mod_{\Bbbk}\rp_{\leqslant n-p}$, and since $n$-coconnective objects are stable under limits this yields that $\lp M_{\bullet}^{\fil}\rp_p$ is $(n-p)$-coconnective for any integer $n$. So $(-)^{\fil}$ preserves both connective and coconnective objects, hence it is $t$-exact.
\end{proof}
To prove \cref{thm:leftcompletion}, we have to show that $\Modmixgr_{\Bbbk}$ is a localization of $\Modfil_{\Bbbk}$, i.e. the $\infinity$-functor $(-)^{\fil}$ is a fully faithful right adjoint whose essential image consists of complete objects.
\begin{propositionn}
	\label{prop:associatedmixedgraded}
	The $\infinity$-functor $(-)^{\fil}\colon\Modmixgr_{\Bbbk}\to\Modfil_{\Bbbk}$ admits a left adjoint.
\end{propositionn}
\begin{proof}
	Being $\Modmixgr_{\Bbbk}$ and $\Modfil_{\Bbbk}$ two $\infinity$-categories which are both presentable and accessible, it is sufficient to show that $(-)^{\fil}$ commutes with all limits and with $\kappa$-filtered colimits for $\kappa$ some regular cardinal. Since $\Modfil_{\Bbbk}$ is an $\infinity$- category of functors, hence limits and colimits are computed weight-wise, it suffices to show that for any integer $p$ the $\infinity$-functor $\left|\sigma_{\leqslant -p}(-)\right|^{\operatorname{t}}$ commutes with limits and is accessible, i.e., it commutes with $\kappa$-filtered colimits for some regular cardinal $\kappa$. Since such $\infinity$-functor is the composition of two \infinity-functors, one of which (namely, $\sigma_{\leqslant -p}$) commutes with limits and is accessible, being a left adjoint, we are left to prove that $\left|-\right|^{\operatorname{t}}$ commutes with limits and is accessible as well.
	\begin{itemize}
		\item The $\infinity$-functor $\left|-\right|^{\operatorname{t}}$ commutes with all limits: this can be seen via direct computations from the description of $\left|-\right|^{\operatorname{t}}$ and by recalling that limits distribute over limits.
		\item The Tate realization $\infinity$-functor is given explicitly by a countable product, so it commutes with $\kappa_2$-filtered colimits for some regular cardinal $\kappa_2$ bigger than $\aleph_0$.
	\end{itemize}
	The claim now follows, once again, from the Adjoint Functor Theorem.
\end{proof}
\cref{prop:associatedmixedgraded} allows us to introduce the following Definition.
\begin{defn}
	\label{def:associatedmixedgraded}
	The $\infinity$-functor $$(-)^{\mixgr}\colon\Modfil_{\Bbbk}\longrightarrow\Modmixgr_{\Bbbk},$$left adjoint to the $\infinity$-functor $$(-)^{\fil}\colon\Modmixgr_{\Bbbk}\to\Modfil_{\Bbbk}$$is the \textit{associated mixed graded $\infinity$-functor}.
\end{defn}
We want to study some properties of the $\infinity$-functor of \cref{def:associatedmixedgraded} as well. One of its key features is described in the following Proposition.
\begin{propositionn}
	\label{prop:gradedweightequivalence}
	Given a filtered $\Bbbk$-module $M_{\bullet}$, the $p$ weight component of its associated mixed graded $\Bbbk$-module $\lp M_{\bullet}\rp^{\mixgr}$ is naturally equivalent to $\Gr_{-p}M_{\bullet}[-2p]$.
\end{propositionn}
\begin{proof}
	Let $$\lp-\rp^{\gr}\colon\Modfil_{\Bbbk}\longrightarrow\Modgr_{\Bbbk}$$ denote the $\infinity$-functor obtained by the composition of $\lp-\rp^{\mixgr}$ with the forgetful $\infinity$-functor $\oblv_{\varepsilon}\colon\Modmixgr_{\Bbbk}\to\Modgr_{\Bbbk}$. Let \begin{align}
	\label{functor:fromfilteredtograded}
	\Gr_{\bullet}\colon\Modfil_{\Bbbk}\longrightarrow\Modgr_{\Bbbk}
	\end{align}denote the $\infinity$-functor obtained by patching the $\infinity$-functors ${\Gr_{-p}}{\lp-\rp}[-2p]\colon\Modfil_{\Bbbk}\longrightarrow\Mod_{\Bbbk}$. In particular, the $p$-th component of this $\infinity$-functor is $\Gr_{-p}\lp-\rp[-2p]$. We will prove that for any filtered ${\Bbbk}$-module $M_{\bullet}$, there is a canonical equivalence of $\Bbbk$-modules ${\Gr_{\bullet}}{ M_{\bullet}}\simeq\lp M_{\bullet}\rp^{\gr}$.\\
	We first observe that $\Gr_{\bullet}\colon \Modfil_{\Bbbk}\to\Modgr_{\Bbbk}$ is the left adjoint to the $\infinity$-functor $$\operatorname{R}^{\fil}\colon\Modgr_{\Bbbk}\longrightarrow\Modfil_{\Bbbk}$$ 
	which sends a graded ${\Bbbk}$-module $N_{\bullet}$ to the filtered ${\Bbbk}$-module whose $p$-th component is $N_{-p}[-2p]$ and whose transition maps are all zero maps (this is essentially \cite[Lemma $3.30$]{enhancingfiltered}, after an homological shift by $-2p$ and the equivalence $\Modgr_{\Bbbk}\simeq \Modgr_{\Bbbk}$ which swaps the graded parts in positive and negative degrees). In the same way, in virtue of \cref{porism:rightadjointforgetful} the $\infinity$-functor which forgets the mixed structure is the left adjoint to the $\infinity$-functor $$\operatorname{R}_{\varepsilon}\colon\Modgr_{\Bbbk}\longrightarrow\Modmixgr_{\Bbbk}.$$Now let $M_{\bullet}$ denote a filtered ${\Bbbk}$-module, and $N_{\bullet}$ denote a graded ${\Bbbk}$-module. By the adjunction $\Gr_{\bullet}\dashv \operatorname{R}^{\fil}$, we have the following equivalences of mapping spaces
	\begin{displaymath}
	{\Map_{\Modgr_{\Bbbk}}}{\lp {\Gr_{\bullet}}{ M_{\bullet}},\hsp N_{\bullet}\rp}\simeq{\Map_{\Modfil_{\Bbbk}}}{\lp M_{\bullet},\hsp {\operatorname{R}^{\fil}}{ N_{\bullet}}\rp}
	\end{displaymath}
	On the other hand, by the adjunction $\oblv_{\varepsilon}\dashv \operatorname{R}_{\varepsilon}$ one has
	\begin{align*}
	{\Map_{\Modgr_{\Bbbk}}}{\lp\lp M_{\bullet}\rp^{\operatorname{gr}},\hsp N_{\bullet}\rp}&\simeq {\Map_{\Modmixgr_{\Bbbk}}}{\lp\lp M_{\bullet}\rp^{\mixgr},\hsp {\operatorname{R}_{\varepsilon}}{ N_{\bullet}}\rp}\\
	&\!\!\!\overset{\ref{def:associatedmixedgraded}}{\simeq}{\Map_{\Modfil_{\Bbbk}}}{\lp M_{\bullet},\hsp \lp {\operatorname{R}_{\varepsilon}}{ N_{\bullet}}\rp^{\fil}\rp}.
	\end{align*}
	By the fully faithfulness of the Yoneda embedding, it is enough to check that $${\operatorname{R}^{\fil}}{ N_{\bullet}}\simeq \lp{\operatorname{R}_{\varepsilon}}{ N_{\bullet}}\rp^{\fil}$$for any graded ${\Bbbk}$-module $N_{\bullet}$. This is seen by direct inspection: in fact, $\lp{\operatorname{R}_{\varepsilon}}{ N_{\bullet}}\rp^{\fil}$ is the filtered ${\Bbbk}$-module whose $p$-th component is obtained by the Tate realization of the naive truncation $\sigma_{\leqslant -p}{\lp{\operatorname{R}_{\varepsilon}}{ N_{\bullet}}\rp}$. By the description of $\operatorname{R}_{\varepsilon}$ provided in \cref{porism:rightadjointforgetful}, $\left|\sigma_{\leqslant -p}{\lp{\operatorname{R}_{\varepsilon}}{ N_{\bullet}}\rp}\right|^{\operatorname{t}}$ is equivalent to the ${\Bbbk}$-module $N_{-p}[-2p]$. The inclusion of $\sigma_{\leqslant -p}{\lp{\operatorname{R}_{\varepsilon}}{ N_{\bullet}}\rp}$ into $\sigma_{\leqslant -p+1}{\lp{\operatorname{R}_{\varepsilon}}{ N_{\bullet}}\rp}$ induces the zero map on the Tate realizations, and so we get exactly the filtered ${\Bbbk}$-module ${\operatorname{R}^{\fil}}{ N_{\bullet}}$ we described above, which completes the proof.
\end{proof}
\begin{remark}
	Working with explicit models, the mixed differential in the mixed graded $\Bbbk$-module $\lp M_{\bullet}\rp^{\mixgr}$ is given by the boundary morphism in the cofiber sequence
	$$\Gr_{p+1}M_{\bullet}\longrightarrow{\cofib}{\lp M_{p+1}\to M_p\rp}\longrightarrow\Gr_p M_{\bullet}$$after a \textit{suitable} (i.e., correctly functorial) choice of the cofibers in each weight. 
\end{remark}
The above characterization of the weight components of the associated mixed graded $\infinity$-functor leads immediately to two important consequences.
\begin{corollaryn}
	\label{corollary:mixedgradedismonoidal}
	The $\infinity$-functor $(-)^{\mixgr}\colon\Modfil_{\Bbbk}\to\Modmixgr_{\Bbbk}$ is strongly monoidal. 
\end{corollaryn}
\begin{proof}
	Let us recall that the \infinity-functor $(-)^{\fil}$ is lax monoidal: it can be checked by straightforward computation, by using the fact that in each piece it is given by the mapping space out of a cocommutative coalgebra object (namely, $\Bbbk(\infinity)$). Being the left adjoint of a lax monoidal \infinity-functor, $(-)^{\mixgr}$ naturally possesses an oplax monoidal structure, and so given two filtered $\Bbbk$-modules $M_{\bullet}$ and $N_{\bullet}$ we have a natural map$$\lp M_{\bullet}\otimesfil_{\Bbbk}N_{\bullet}\rp^{\mixgr}\longrightarrow\lp M_{\bullet}\rp^{\mixgr}\otimes_{\Bbbk}^{\mixgr}\lp N_{\bullet}\rp^{\mixgr}.$$ Since equivalences of mixed graded $\Bbbk$-modules are detected by the underlying graded $\Bbbk$-modules and forgetting the mixed structure is a strongly monoidal \infinity-functor, it suffices to show that the underlying graded map $$\oblv_{\varepsilon}\lp M_{\bullet}\otimesfil_{\Bbbk}N_{\bullet}\rp^{\mixgr}\simeq \lp M_{\bullet}\otimesfil_{\Bbbk}N_{\bullet}\rp^{\gr}\longrightarrow\oblv_{\varepsilon}\lp\lp M_{\bullet}\rp^{\mixgr}\otimes_{\Bbbk}^{\mixgr}\lp N_{\bullet}\rp^{\mixgr}\rp\simeq \lp M_{\bullet}\rp^{\gr}\otimes_{\Bbbk}^{\gr}\lp N_{\bullet}\rp^{\gr}$$ is an equivalence. Now, \cref{prop:gradedweightequivalence} tells us that the graded $\Bbbk$-module on the left is described in its $p$ weight component by the formula$$\lp M_{\bullet}\otimesfil_{\Bbbk} N_{\bullet}\rp^{\mixgr}_{\!p}\simeq \Gr_{-p}\lp M_{\bullet}\otimesfil_{\Bbbk}N_{\bullet}\rp.$$But now \cref{lemma:scholze5} tells that the above object is equivalent to $$\bigoplus_{i+j=p}\Gr_{-i}M_{\bullet}\otimes_{\Bbbk}\Gr_{-j}N_{\bullet}$$which is exactly the $p$ weight component of the tensor product of $\lp M_{\bullet}\rp ^{\gr}\otimes^{\gr}_{\Bbbk}\lp N_{\bullet}\rp ^{\gr}$. So, the natural map above is indeed an equivalence.
\end{proof}
\begin{corollaryn}
	\label{cor:mixedgradedistexact}
	The $\infinity$-functor $(-)^{\mixgr}\colon\Modfil_{\Bbbk}\to\Modmixgr_{\Bbbk}$ is $t$-exact.
\end{corollaryn}
\begin{proof}
	Let $M_{\bullet}$ be an $n$-connective filtered ${\Bbbk}$-module. Then $\Gr_pM_{\bullet}$ is an $(n-p)$-connective ${\Bbbk}$-module for any integer $p$, therefore $$\lp M_{\bullet}\rp^{\mixgr}_{\!p}\overset{\text{Prop. }\ref{prop:associatedmixedgraded}}{\simeq}\Gr_{-p}M_{\bullet}[-2p]$$is $(n-p)$-connective.\\
	On the converse, if $M_{\bullet}$ is an $n$-coconnective filtered ${\Bbbk}$-module, then $ M_p$ is an $(n-p)$-coconnective ${\Bbbk}$-module for any $p$. This implies that $\Gr_pM_{\bullet}$ is an $(n-p)$-coconnective ${\Bbbk}$-module for any $p$ as well: in fact, $\Gr_pM_{\bullet}$ is the cofiber of $M_{p+1}\longrightarrow M_p$. Since $M_{\bullet}$ is $n$-coconnective, $M_{p+1}$ and $M_p$ are respectively $(n-p-1)$-coconnective and $(n-p)$-coconnective ${\Bbbk}$-modules. By inspecting the long exact sequence induced on the homotopy groups by the cofiber sequence
	\begin{displaymath}
	M_{p+1}\longrightarrow M_p\longrightarrow\Gr_pM_{\bullet}
	\end{displaymath}
	one has that $\pi_q{\Gr_p}{M_{\bullet}}\cong 0$ for all $q> n-p$, hence ${\Gr_p}{M_{\bullet}}$ is an $(n-p)$-coconnective ${\Bbbk}$-module and therefore $$\lp M_{\bullet}\rp^{\mixgr}_{\!p}\overset{\text{Prop. }\ref{prop:associatedmixedgraded}}{\simeq} \Gr_{-p}M_{\bullet}[-2p]$$is $(n-p)$-coconnective for all $p$'s.
\end{proof}
The following two claims are the last ingredients we need to prove \cref{thm:leftcompletion}.
\begin{propositionn}
	\label{prop:mixedgradedequivalence}
	The counit morphism $ (-)^{\mixgr}\circ (-)^{\fil}\to\operatorname{id}_{ \Modmixgr_{\Bbbk}}$ of the adjunction $(-)^{\mixgr}\dashv (-)^{\fil}$ is an equivalence of $\infinity$-functors.
\end{propositionn}
\begin{proof}
	Let $M_{\bullet}$ be a mixed graded $\Bbbk$-module. By the description provided in \cref{remark:keyremarkfiltration}, its associated filtered $\Bbbk$-module is$$\dots\longhookrightarrow\prod_{q\geqslant p+1}M_{-q}[-2q]\longhookrightarrow\prod_{q\geqslant p}M_{-q}[-2q]\longhookrightarrow\prod_{q\geqslant p-1}M_{-q}[-2q]\longhookrightarrow\dots$$and for any $p$, $\Gr_{p}\lp M_{\bullet}\rp^{\fil}[-2p]\simeq M_{-p}[-2p]$. Applying $(-)^{\mixgr}$ to $\lp M_{\bullet}\rp^{\fil}$, we get a mixed graded $\Bbbk$-module whose $p$-th weight component is equivalent to $$\Gr_{-p}\lp M_{\bullet}\rp^{\fil}[-2p]\overset{\text{Prop. }\ref{prop:associatedmixedgraded}}{\simeq} M_{-(-p)}[-2(-p)][-2p]\simeq M_p.$$Since equivalences of mixed graded $\Bbbk$-modules are detected by forgetting the mixed structure, we have that the counit is an equivalence because it induces on each weights precisely the equivalence $\Gr_{-p}{\lp M_{\bullet}\rp}^{\fil}[-2p]\simeq M_p $.
\end{proof}
\begin{propositionn}
	\label{prop:completefiltered}
	When restricted to $n$-coconnective objects of $\Modfil_{\Bbbk}$, the unit morphism $\eta\colon \operatorname{id}_{\lp\Modfil_{\Bbbk}\rp_{\leqslant n}}\to (-)^{\fil}\circ (-)^{\mixgr}$ of the adjunction $(-)^{\mixgr}\dashv (-)^{\fil}$ is an equivalence of $\infinity$-functors for all integers $n$.
\end{propositionn}
\begin{proof}
	Let $M_{\bullet}$ be a filtered $\Bbbk$-module which is $n$-coconnective for the Beilinson $t$-structure for some integer $n$. We briefly study what $\lp\lp N_{\bullet}\rp^{\mixgr}\rp^{\!\fil}$ is: its $p$-th part is the $\Bbbk$-module given by the Tate realization
	\begin{displaymath}
	\left|\sigma_{\leqslant -p}\lp N_{\bullet}\rp^{\mixgr}\right|^{\operatorname{t}}\simeq\prod_{q\geqslant p}\lp N_{\bullet}\rp^{\mixgr}_{-q}[-2q].
	\end{displaymath}
	But we have that
	\begin{align*}\lp N_{\bullet}\rp^{\mixgr}_{-q}\overset{\text{Prop. }\ref{prop:associatedmixedgraded}}{\simeq}{\Gr_{-(-q)}}{N_{\bullet}}[-2(-q)]\simeq {\Gr_q}{N_{\bullet}}[2q]
	\end{align*}
	and since the shift $\infinity$-functors commute with arbitrary products we have that\begin{displaymath}
	\lp\lp N_{\bullet}\rp^{\mixgr}\rp^{\!\fil}_{\!p}\simeq \prod_{q\geqslant p}{\Gr_q}N_{\bullet}
	\end{displaymath}The transition maps are just given by inclusions.\\
	Let\begin{displaymath}
	F_p\coloneqq{\fib}{\lp N_p\overset{\eta_p}{\longrightarrow} \prod_{q\geqslant p}{\Gr_q}{N_{\bullet}}\rp}
	\end{displaymath}denote the fiber of the $p$-th component of $\eta$. Then we have the following diagram, where every row and every column is a fiber sequence.
	\begin{displaymath}
	\begin{tikzpicture}[scale=0.65]
	\node (fib1) at (-7,4){$F_{p+1}$};
	\node (fib2) at (0,4){$F_p$};
	\node (fib3) at (7,4){${\cofib}{\lp F_{p+1}\longrightarrow F_p\rp}$};
	\node (Nn1) at (-7,2){$N_{p+1}$};
	\node (Nn) at (0,2){$N_p$};
	\node (Grn) at (7,2){${\Gr_p}{N_{\bullet}}$};
	\node (prodn1) at (-7,-0.3){$\underset{q\geqslant p+1}{\prod}{\Gr_q}{N_{\bullet}}$};
	\node (prodn) at (0,-0.3){$\underset{q\geqslant p}{\prod}\hsp{\Gr_q}{N_{\bullet}}$};
	\node (Grnb) at (7,-0.1){${\Gr_p}{N_{\bullet}}$};
	\draw[->,font=\scriptsize] (Nn1) to node[left]{$\eta_{p+1}$} (prodn1);
	\draw[->,font=\scriptsize] (Nn) to node[right]{$\eta_p$} (prodn);
	\draw[->,font=\scriptsize] (Grn) edge node[above,rotate=-90]{$\simeq$} (Grnb);
	\draw[->] (fib1) -- (fib2);
	\draw[->] (fib1) -- (Nn1);
	\draw[->] (fib2) -- (fib3);
	\draw[->] (fib3) -- (Grn);
	\draw[->] (fib2) -- (Nn);
	\draw[->] (Nn1) -- (Nn);
	\draw[->] (Nn) -- (Grn);
	\draw[right hook->] (-5.2,0) -- (-1.5,0);
	\draw[->] (1.6,0) -- (6,0);
	\end{tikzpicture}
	\end{displaymath}
	The map on the bottom right is an equivalence, and this forces ${\cofib}{\lp F_{p+1}\to F_p\rp}$ to be zero. This implies that $F_{p+1}\to F_p$ is an equivalence; by induction, we can conclude that $F_p$ is the fiber of $$N_m\longrightarrow\prod_{q\geqslant m}\Gr_qN_{\bullet}$$for any integer $m$. So, the fact that $\eta_m$ is an equivalence for any $m$ is equivalent to proving that $F_p$ is $0$, but this is a consequence of Lemma \ref{lemma:coconnectiveimpliescomplete}. Indeed, if $F_p$ was not $0$, we would have a non-trivial $\Bbbk$-module with a non-trivial map $F_p\to N_m$ for all $m$'s. Such map would be forced to factor through $N_{\infty}$; in particular, it  would yield that $N_{\infty}$ is non-trivial, hence that $N_{\bullet}$ is not complete. But this contradicts the (eventually) coconnectivity assumption on $N_{\bullet}$.
\end{proof}
\begin{remark}
	In the proof of \cref{prop:completefiltered}, we did not need to restrict ourselves to eventually coconnective objects. In fact, the key property of eventually coconnective filtered $\Bbbk$-modules that we used was their completeness (\cref{lemma:coconnectiveimpliescomplete}): the unit morphism of the adjunction $\lp -\rp^{\mixgr}\dashv\lp-\rp^{\fil}$ is an equivalence on \textit{all} complete filtered $\Bbbk$-modules, which are not necessarily eventually coconnective (take for example the filtered $\Bbbk$-module given by the $0$ sequence everywhere, except for $\Bbbk[1]$ in weight $0$). The fiber$$F_p\coloneqq\fib\lp N_p\to \lp\lp N_{\bullet}\rp^{\mixgr}\rp^{\fil}_p\rp$$of the proof of \cref{prop:completefiltered} is, in fact, $N_{\infty}$ itself: the unit $\eta$ fails to be an equivalence \textit{precisely} on non-complete filtered $\Bbbk$-modules.
\end{remark}
\begin{parag}
	All previous propositions, lemmas and remarks imply together \cref{thm:leftcompletion}. In fact, \cref{prop:leftcompletionBeilinson} allows us to identify the left completion $\Modfilcomp_{\Bbbk}$ with the $\infinity$-category of complete filtered $\Bbbk$-modules in the sense of \cref{def:completefiltered}, endowed with the restriction of the Beilinson $t$-structure. The equivalence of \cref{thm:leftcompletion} is then provided by the composition of the natural inclusion $\Modfilcomp_{\Bbbk}\subseteq\Modfil_{\Bbbk}$ with the associated mixed graded $\infinity$-functor of \cref{def:associatedmixedgraded}. Its inverse is described by the composition of the associated filtered $\infinity$-functor of \cref{def:filteredassociated} with the completion $\infinity$-functor of \cref{lemma:scholze2}. The fact that these $\infinity$-functors are indeed equivalences follows from Propositions \ref{prop:mixedgradedequivalence} and \ref{prop:completefiltered}. The $t$-exactness of the equivalence $\Modfilcomp_{\Bbbk}\simeq \Modmixgr_{\Bbbk}$ follows from the $t$-exactness of the $\infinity$-functors $(-)^{\fil}$ and $(-)^{\mixgr}$.
\end{parag}
\begin{parag}
	Finally, since $(-)^{\fil}$ and $(-)^{\mixgr}$ are both $t$-exact $\infinity$-functors which induce equivalences on coconnective objects, it follows that the hearts of the two $t$-structures are naturally equivalent and both yield, when $\Bbbk$ is a discrete ring, the usual $1$-category of chain complexes of $\Bbbk$-modules: this is the last statement that we left unproven in \cref{thm:mixedgradedtstructure}.
\end{parag}	
\begin{porism}
	\label{porism:monoidalstructurecompletefiltered}
	The strongly monoidal structure of the \infinity-functor $\lp-\rp^{\mixgr}$ proved in \cref{corollary:mixedgradedismonoidal}, together with the description of the monoidal structure on complete filtered $\Bbbk$-modules provided by \cref{lemma:scholze4}, implies that the equivalence of \cref{thm:leftcompletion} is \textit{strongly monoidal}. Explicitly, given two complete filtered $\Bbbk$-modules $M_{\bullet}$ and $N_{\bullet}$, their tensor product in $\Modfilcomp_{\Bbbk}$, given by the completion of their tensor product as mere filtered $\Bbbk$-modules (\cite[Theorem $2.2.5$]{enhancingfiltered}), is equivalently described as
	\begin{align*}
	M_{\bullet}\text{$\widehat{\otimes}^{\fil}_{\Bbbk}$}N_{\bullet}
	%&\simeq \lp\lp M_{\bullet}\rp^{\mixgr}\rp^{\fil}\hat{\otimes}^{\fil}_{\Bbbk} \lp\lp N_{\bullet}\rp^{\mixgr}\rp^{\fil}
	\simeq \lp\lp M_{\bullet}\rp^{\mixgr}\otimesmixgr_{\Bbbk}\lp  N_{\bullet}\rp^{\mixgr}\rp^{\fil}.
	\end{align*}
\end{porism}
\begin{porism}
	\label{porism:taterealizationstronglymonoidal}
	In \cref{sec:basicmixedgraded}, we mentioned how the Tate realization \infinity-functor was \textit{strongly} monoidal when restricted to non-negatively mixed graded $\Bbbk$-modules. This claim follows immediately from \cref{thm:leftcompletion}: indeed, we have an equivalence in $\Fun{\lp\Modmixgr_{\Bbbk},\hsp\Mod_{\Bbbk}\rp}$ $$\left|-\right|^{\operatorname{t}}\simeq (-)_{-\scriptstyle\infty}\circ(-)^{\fil}$$and since $(-)^{\mixgr}$ is a localization \infinity-functor, we have another equivalence $$\left|-\right|^{\operatorname{t}}\circ (-)^{\mixgr}\simeq (-)_{-\scriptstyle\infty}$$if one restricts $(-)^{\mixgr}$ to $\Modfilcomp_{\Bbbk}$. Since the \infinity-functor $(-)_{-\scriptstyle\infty}$ is strongly monoidal (\cite[Section $2.23$]{enhancingfiltered}), one simply has to show that the tensor product of two complete filtered $\Bbbk$-modules $M_{\bullet}$ and $N_{\bullet}$ such that $M_p$ and $N_p$ are zero for all positive integers $p$ is again complete. The explicit formula for the filtered tensor product provided in \ref{parag:monoidalstructureonmodfil} yields that$$\lp M_{\bullet}\otimesfil_{\Bbbk} N_{\bullet}\rp_n\simeq \underset{p+q\geqslant n}{\colim}\hsp M_p\otimes_{\Bbbk} N_q.$$If both $M_{\bullet}$ and $N_{\bullet}$ are endowed with a filtration bounded in non-positive weights, the above formula is $0$ for any positive integer $n$. Hence the limit on the tower of $\Bbbk$-modules corresponding to $M_{\bullet}\otimesfil_{\Bbbk}N_{\bullet}$ is zero, hence the filtered tensor product in this case preserves completeness.
\end{porism}
\begin{porism}
	\label{porism:taterealizationcolimits}
	Arguing analogously to \cref{porism:taterealizationstronglymonoidal}, we can prove that the Tate realization \infinity-functor preserves colimits when it is restricted to mixed non-negatively graded $\Bbbk$-modules. Indeed, we have a commutative square
	$$\begin{tikzpicture}[scale=0.75]
	\node (a) at (-2,2){$\Modmixgrcn_{\Bbbk}$};
	\node (b) at (2,2){$\operatorname{Mod}^{\fil,\leqslant0}_{\Bbbk}$};
	\node (c) at (-2,0){$\Modmixgr_{\Bbbk}$};
	\node (d) at (2,0){$\Modfil_{\Bbbk}$};
	\node (e) at (0,-2){$\Mod_{\Bbbk}$};
	\draw[right hook->, font=\scriptsize](a) -- (c);
	\draw[right hook->, font=\scriptsize](b) -- (d);
	\draw[->, font=\scriptsize](a) to node[above]{$(-)^{\fil}$} (b);
	\draw[->, font=\scriptsize](c) to node[above]{$(-)^{\fil}$} (d);
	\draw[->, font=\scriptsize](c) to[bend right] node[ left]{$\left|-\right|^{\operatorname{t}}$} (e);
	\draw[->, font=\scriptsize](d) to[bend left] node[right]{$(-)_{-\scriptstyle\infty}$} (e);
	\end{tikzpicture}$$
	where $\operatorname{Mod}^{\fil,\leqslant0}_{\Bbbk}$ is the \infinity-functor \infinity-category ${\Fun}{\lp\NN_{\geqslant},\hsp \Mod_{\Bbbk}\rp}$ of $\Bbbk$-modules with filtration bounded in non-positive weights, and the inclusion is induced by the map of posets $\NN_{\geqslant}\subseteq \ZZ_{\geqslant}$. The \infinity-functor $(-)_{-\scriptstyle\infty}$ commutes with colimits because it is a colimit \infinity-functor itself, so it suffices to show that $(-)^{\fil}\colon\Modmixgrcn_{\Bbbk}\to \operatorname{Mod}^{\fil,\leqslant0}_{\Bbbk}$ commutes with colimits. This is true because colimits of non-positively filtered $\Bbbk$-modules are again non-positively filtered, since colimits are computed weight-wise. Hence they are again complete, and so the colimit in mixed graded $\Bbbk$-modules agrees with the colimit in filtered $\Bbbk$-modules. In particular, the \infinity-functor\begin{align}
	\label{functor:taterealizationconnective}
	\left|-\right|^{\operatorname{t}}\colon\Modmixgrcn_{\Bbbk}\longrightarrow\Mod_{\Bbbk}
	\end{align}admits a right adjoint, which we can compute explicitly using the fact that the \infinity-functor \ref{functor:taterealizationconnective} is the composition of$$\Modmixgrcn_{\Bbbk}\overset{\simeq}{\longrightarrow}\operatorname{Mod}^{\fil,\leqslant0}_{\Bbbk}\longhookrightarrow\Modfil_{\Bbbk}\xrightarrow{(-)_{-\scriptstyle\infty}}\Mod_{\Bbbk}.$$In fact, each \infinity-functor of this composition is a left adjoint, so the right adjoint to \ref{functor:taterealizationconnective} is the composition of right adjoints$$\Mod_{\Bbbk}\overset{(-)^{\operatorname{const}}}{\longrightarrow}\Modfil_{\Bbbk}\overset{\sigma_{\leqslant 0}}{\longrightarrow}\operatorname{Mod}^{\fil,\leqslant0}_{\Bbbk}\overset{\simeq}{\longrightarrow}\Modmixgrcn_{\Bbbk}$$
	where $(-)^{\operatorname{const}}$ is the diagonal \infinity-functor sending a $\Bbbk$-module to the constant sequence, $\sigma_{\leqslant 0}$ is the naive truncation \infinity-functor for filtered $\Bbbk$-modules (built similarly as the \infinity-functor of Definition \ref{def:naivetruncationindegleqp}), and the last \infinity-functor is the mixed graded $\Bbbk$-module construction (which is an equivalence because every filtered $\Bbbk$-module is complete if it is $0$ for all positive integers). In particular, this right adjoint agrees with$$(-)(0)\colon\Mod_{\Bbbk}\longrightarrow\Modmixgrcn_{\Bbbk}\subseteq\Modmixgr_{\Bbbk}.$$
\end{porism}
\printbibliography
%\begin{remark}
%The \infinity-functor $(-)(0)\colon\Mod_{\Bbbk}\to\Modmixgr_{\Bbbk}$ always preserves all limits and colimits, hence by presentability always admits a left adjoint given by sending $M_{\bullet}$ to the cofiber $\cofib{\lp M_1\xrightarrow{\varepsilon_1}M_0[-1]\rp}$. The fact that for mixed non-negatively graded $\Bbbk$-modules this left adjoint agrees with $|-|^{\operatorname{t}}$ is a consequence that, in terms of explicit resolutions, such fiber is computed by 
%\end{remark}
\end{document}